%% file: messenger01.tex
\documentclass[reqno]{amsart}
\usepackage[latin9]{inputenc}
\usepackage{geometry}
\usepackage{hyperref}
\usepackage{xcolor}
\usepackage{pdfcolmk}
\usepackage{float}
\usepackage{mathtools}
\usepackage{amsbsy}
\usepackage{amstext}
\usepackage{amsthm}
\usepackage{amssymb}
\usepackage{graphicx}
\PassOptionsToPackage{normalem}{ulem}
\usepackage{ulem}
\makeatletter

\providecommand{\tabularnewline}{\\}
\providecolor{lyxadded}{rgb}{0,0,1}
\providecolor{lyxdeleted}{rgb}{1,0,0}

\DeclareRobustCommand{\lyxsout}[1]{\ifx\\#1\else\sout{#1}\fi}

\numberwithin{equation}{section}
\numberwithin{figure}{section}
\theoremstyle{plain}
\newtheorem{thm}{\protect\theoremname}
  \theoremstyle{remark}
  \newtheorem{rem}[thm]{\protect\remarkname}
  \theoremstyle{plain}
  \newtheorem{lem}[thm]{\protect\lemmaname}


\usepackage{stmaryrd}

\makeatother

  \providecommand{\lemmaname}{Lemma}
  \providecommand{\remarkname}{Remark}
\providecommand{\theoremname}{Theorem}

\begin{document}

\title{Weak SINDy: Galerkin-Based Data-Driven Model Selection}

\author{Daniel A. Messenger and David M. Bortz}

\email{daniel.messenger@colorado.edu, dmbortz@colorado.edu, \thanks{Department of Applied Mathematics, University of Colorado, Boulder,
CO 80309-0526, USA.}}
\begin{abstract}
We present a novel weak formulation and discretization for discovering
general equations from noisy measurement data. This method of learning
differential equations from data fits into a new class of algorithms
that replace pointwise derivative approximations with linear transformations
and a variance reduction techniques. Our approach improves on the
standard SINDy algorithm presented in \cite{brunton2016discovering}
by orders of magnitude. We first show that in the noise-free regime,
this so-called Weak SINDy (WSINDy) framework is capable of recovering
the dynamic coefficients to very high accuracy, with the number of
significant digits equal to the tolerance of the data simulation scheme.
Next we show that the weak form naturally accounts for white noise
by identifying the correct nonlinearities with coefficient error scaling
favorably with the signal-to-noise ratio while significantly reducing
the size of linear systems in the algorithm. In doing so, we combine
the ease of implementation of the SINDy algorithm with the natural
noise-reduction of integration as demonstrated in \cite{schaeffer2017sparse}
to arrive at a more robust and practical method of sparse recovery
that correctly identifies systems in both small-noise and large-noise
regimes.
\end{abstract}

\maketitle
\textbf{\small{}Keywords: }{\small{}data-driven model selection, nonlinear
dynamics, sparse recovery, generalized least squares, Galerkin method,
adaptive grid, white noise}{\small\par}

\section{Problem Statement}

Consider a first-order dynamical system in $D$ dimensions of the
form 
\begin{equation}
\frac{d}{dt}\mathbf{x}(t)=\mathbf{F}(\boldsymbol{x}(t)),\quad\mathbf{x}(0)=\mathbf{x}_{0}\in\mathbb{R}^{D},\quad0\leq t\leq T,\label{ode}
\end{equation}
and measurement data $\mathbf{y}\in\mathbb{R}^{M\times D}$ given
at $M$ timepoints $\mathbf{t}=[t_{1},\,\dots\,,t_{M}]$ by 
\[
\mathbf{y}_{md}=\mathbf{x}_{d}(t_{m})+\epsilon_{md},\qquad m\in[M],\ d\in[D],
\]
where throughout we use the bracket notation $[M]:=\{1,\dots,M\}$.
The matrix $\epsilon\in\mathbb{R}^{M\times D}$ represents i.i.d.
measurement noise. The focus of this article is the reconstruction
of the dynamics \eqref{ode} from the measurements $\mathbf{y}$.

The SINDy algorithm (Sparse Identification of Nonlinear Dynamics \cite{brunton2016discovering})
has been shown to be successful in solving this problem for sparsely
represented nonlinear dynamics when noise is small and dynamic scales
do not vary across multiple orders of magnitude. This framework assumes
that the function $\mathbf{F}:\mathbb{R}^{D}\to\mathbb{R}^{D}$ in
\eqref{ode} is given component-wise by 
\begin{equation}
\mathbf{F}_{d}(\mathbf{x}(t))=\sum_{j=1}^{J}\mathbf{w^{\star}}_{jd}\,f_{j}(\mathbf{x}(t))\label{candbasis}
\end{equation}
for some known family of functions $(f_{j})_{j\in[J]}$ and a sparse
weight matrix $\mathbf{w}^{\star}\in\mathbb{R}^{J\times D}$. The
problem is then transformed into solving for $\mathbf{w}^{\star}$
by building a data matrix $\Theta(\mathbf{y})\in\mathbf{\mathbb{R}}^{M\times J}$
given by 
\[
\Theta(\mathbf{y})_{mj}=f_{j}(\mathbf{y}_{m}),\qquad\mathbf{y}_{m}=(\mathbf{y}_{m1},\dots,\mathbf{y}_{mD}),
\]
so that the candidate functions are directly evaluated at the noisy
data. Solving \eqref{ode} for $\mathbf{F}$ then reduces to solving
\begin{equation}
\dot{\mathbf{y}}=\Theta(\mathbf{y})\,\widehat{\mathbf{w}}\label{discSINDy}
\end{equation}
for a sparse weight matrix $\widehat{\mathbf{w}}$, where $\dot{\mathbf{y}}$
is the numerical time derivative of the data $\mathbf{y}$. Sequentially-thresholded
least squares is then used to arrive at a sparse solution.\\

The automatic creation of an accurate mathematical model from data
is a challenging task and research into statistically rigorous model
selection can be traced back to Akaike's seminal work in the 1970's
\cite{Akaike1974IEEETransAutomControl,Akaike1977Applicationsofstatistics}.
In the last 20 years, there has been substantial work in this area
at the interface between applied mathematics and statistics (see \cite{BortzNelson2006BullMathBiol,LagergrenNardiniMichaelLavigneEtAl2020ProcRSocA,LillacciKhammash2010PLoSComputBiol,ToniWelchStrelkowaEtAl2009JRSocInterface,WarneBakerSimpson2019BullMathBiola,WuWu2002StatistMed}
for both theory and applications). More recently, the formulation
of system discovery problems in terms of a candidate basis of nonlinear
functions \eqref{candbasis} and subsequent discretization \eqref{discSINDy}
was introduced in \cite{wang2011predicting} in the context of catastrophe
prediction. The authors of \cite{wang2011predicting} used compressed
sensing techniques to enforce sparsity. Since then there has been
an explosion of interest in the problem of identifying nonlinear dynamical
systems from data, with some of the primary techniques being Gaussian
process regression \cite{raissi2017machine}, deep neural networks
\cite{rudy2019deep}, Bayesian inference \cite{zhang2018robust,zhang2019robust}
and a variety of methods from numerical analysis \cite{kang2019ident,keller2019discovery}.
These techniques have been successfully applied to discovery of both
ordrinary and partial differential equations. The variety of approaches
qualitatitively differ in the interpretability of the resulting data-driven
dynamical system, the practicality of the algorithm, and the robustness
due to noise, scale separation, etc. For instance, a neural-network
based data-driven dynamical system does not easily lend itself to
physical interpretation. As well, certain sparsification techniques
are not practical to the general scientific community, where the problem
of system identification from data is ubiquitous. The SINDy algorithm
allows for direct interpretations of the dynamics from identified
differential equations and uses sequentially thresholded least-squares
to enforce sparsity, which is not nearly as robust as other approaches
but is easy to implement and has been proven to converge to sparse
local minimizers in \cite{zhang2019convergence}. Therefore, for simplicity
we use sequential thresholding in this article to demonstrate the
viability of our proposed weak formulation. Naturally one could investigate
using a more robust sparsification strategy.\\

The aim of the present article is to provide rigorous justification
for using the weak formulation of the dynamics in place of local pointwise
derative approximations, as well as a robust algorithm for doing so.
As such, we restrict numerical experiments to autonomous ordinary
differential equations for their immenability to analysis. Natural
next steps are to explore identification of PDEs and non-autonomous
dynamical systems. We note that the use of integral equations for
system identification was introduced in \cite{schaeffer2017sparse},
where compressed sensing techniques were used to enforce sparsity,
and that this technique can be seen as a special case of the method
introduced here. In Section 2 we introduce the algorithm with analysis
of the resulting error structure and in Section 3 we provide numerical
experimentation with a range of nonlinear systems. In Section 4, we
provide concluding remarks including a brief comparison between WSINDy
and conventional SINDy as well as natural next directions for this
line of research.\\

\section{Weak SINDy}

We approach the problem of system identification \eqref{discSINDy}
from a non-standard perspective by utilizing the weak form of the
differential equation. Recall that for any smooth test function $\phi:\mathbb{R}\to\mathbb{R}$
(absolutely continuous is sufficient) and interval $(a,b)\subset[0,T]$,
equation \eqref{ode} admits the weak formulation 
\begin{equation}
\phi(b)\mathbf{x}(b)-\phi(a)\mathbf{x}(a)-\int_{a}^{b}\phi'(u)\,\mathbf{x}(u)\,du=\int_{a}^{b}\phi(u)\,\mathbf{F}(\mathbf{x}(u))\,du,\qquad0\leq a<b\leq T.\label{IBP}
\end{equation}
With $\phi=1$, we arrive at the integral equation of the dynamics
explored in \cite{schaeffer2017sparse}. If we instead take $\phi$
to be non-constant and compactly supported in $(a,b)$, we arrive
at 
\begin{equation}
-\int_{a}^{b}\phi'(u)\,\mathbf{x}(u)\,du=\int_{a}^{b}\phi(u)\,\mathbf{F}(\mathbf{x}(u))\,du.\label{phiI}
\end{equation}
We then define the generalized residual $\mathcal{R}(\mathbf{w};\phi)$
for a given test function by replacing $\mathbf{F}$ with a candidate
element from the span of $(f_{j})_{j\in[J]}$ and $\mathbf{x}$ with
$\mathbf{y}$ as follows: 
\begin{equation}
\mathcal{R}(\mathbf{w};\phi):=\int_{a}^{b}\left(\phi'(u)\,\mathbf{y}(u)+\phi(u)\,\left(\sum_{j=1}^{J}\mathbf{w}_{j}\,f_{j}(\mathbf{y}(u))\right)\right)\,du.\label{gres}
\end{equation}
Clearly with $\mathbf{w}=\mathbf{w}^{\star}$ and $\mathbf{y}=\mathbf{x}(t)$
we have $\mathcal{R}(\mathbf{w};\phi)=0$ for all $\phi$ compactly-supported
in $(a,b)$; however, $\mathbf{y}$ is a discrete set of data, hence
\eqref{gres} can at best be approximated numerically, with measurement
noise presenting a significant barrier to accurate indentification
of $\mathbf{w}^{\star}$.

\subsection{Method Overview}

For analogy with traditional Galerkin methods, consider the forward
problem of solving a dynamical system such as \eqref{ode} for $\mathbf{x}$.
The Galerkin approach is to seek a solution $\mathbf{x}$ represented
in a chosen trial basis $(f_{j})_{j\in[J]}$ such that the residual
$\mathcal{R}$, defined by 
\[
\mathcal{R}=\int\phi(t)(\dot{\mathbf{x}}(t)-\mathbf{F}(\mathbf{x}(t)))\,dt,
\]
is minimized over all test functions $\phi$ living in the span of
a given test function basis $(\phi_{k})_{k\in[K]}$. If the trial
and test function bases are known analytically, inner products of
the form $\langle f_{j},\phi_{k}\rangle$ appearing in the residual
can be computed exactly. Thus, the computational error results only
from representing the solution in a finite-dimensional function space.
\\

The method we present here can be considered a data-driven Galerkin
method of solving for $\mathbf{F}$ where the trial ``basis" is
given by the set of gridfunctions $(f_{j}(\mathbf{y}))_{j\in[J]}$
evaluated at the data and only the test-function basis $(\phi_{k})_{k\in[K]}$
is known analytically. In this way, inner products appearing in $\mathcal{R}(\mathbf{w};\phi)$
must be approximated numerically, implying that the accuracy of the
recovered weights $\widehat{\mathbf{w}}$ is ultimately limited by
the quadrature scheme used to discretize inner products. Using Lemma
\ref{traplemm} below, we show that the correct coefficients $\mathbf{w}^{\star}$
may be recovered to effective machine precision accuracy (given by
the tolerance of the forward ODE solver) from noise-free trajectories
$\mathbf{y}$ by discretizing \eqref{phiI} using the trapezoidal
rule and choosing $\phi$ to decay smoothly to zero at the boundaries
of its support. In this article we demonstrate this fact by choosing
test functions from a particular family of unimodal piece-wise polynomials
$\mathcal{S}$ defined in \eqref{testfcn}.\\

Having chosen a quadrature scheme, the next accuracy barrier is presented
by measurement noise, which introduces a bias in the weights. Below
we analyze the distribution of the residuals $\mathcal{R}(\mathbf{w};\phi)$
to arrive at a generalized least squares approach where an approximate
covariance matrix can be computed directly from the test functions.
This analysis also shows that placing test functions near steep gradients
in the dynamics improves recovery, hence we develop a self-consistent
and stable algorithm for constructing a test function basis adaptively
near these regions which also does not rely on pointwise approximation
of derivatives. Overall, we show that when noise is present, our method
produces a recovered weight matrix $\widehat{\mathbf{w}}$ with the
number of significant digits scaling optimally with the signal-to-noise
ratio $\sigma_{SNR}$ (defined below).\\

\begin{rem}
The weak formulation of the dynamics introduces a wealth of information:
given $M$ timepoints $\mathbf{t}=(t_{m})_{m\in[M]}$, equation \eqref{phiI}
affords $K=M(M-1)/2$ residuals over all possible supports $(a,b)\subset\mathbf{t}\times\mathbf{t}$
with $a<b$. Of course, one could also assimilate the responses of
multiple families of test functions $\left(\{\phi_{k}^{1}\}_{k\in[K_{1}]},\{\phi_{k}^{2}\}_{k\in[K_{2}]},\dots\right)$;
however, the computational complexity of such an exhaustive approach
quickly becomes intractable. We stress that even with large noise,
our proposed method identifies the correct nonlinearities with accurate
weight recovery while keeping the number of test functions much lower
than the number of timepoints ($K\ll M$).
\end{rem}

\subsection{Algorithm: Weak SINDy}

We state here the Weak SINDy algorithm in full generality. We propose
a generalized least squares approach with approximate covariance matrix
$\Sigma$. Below we derive a particular choice of $\Sigma$ which
utilizes the action of the test functions $(\phi_{k})_{k\in[K]}$
on the data $\mathbf{y}$. Sequential thresholding on the weight coefficients
$\mathbf{w}$ with thresholding parameter $\lambda$ is used to enforce
sparsity. In addition, an $\ell_{2}$-regularization term with coefficient
$\gamma$ is included for problems involving rank deficiency. Methods
of choosing optimal $\lambda$ and $\gamma$ are not included in this
study. We note that $\lambda<\min_{\mathbf{w}^{\star}\neq0}|\mathbf{w}^{\star}|$
is necessary for recovery and that with low noise our method is not
sensitive to $\lambda$. Throughout we mostly set $\gamma=0$, however
some problems do require regularization, such as the Lotka-Volterra
system. \\

\noindent $\widehat{\mathbf{w}}=\textbf{WSINDy}\left(\mathbf{y},\mathbf{t}\ ;\ (\phi_{k})_{k\in[K]},\,(f_{j})_{j\in[J]},\,\Sigma,\lambda,\gamma\right)$: 
\begin{enumerate}
\item Construct matrix of trial gridfunctions $\Theta(\mathbf{y})=\big[f_{1}(\mathbf{t},\mathbf{y})\,|\,\dots\,|\,f_{J}(\mathbf{t},\mathbf{y})\big]$ 
\item Construct integration matrices $\mathbf{V}$, $\mathbf{V}'$ such
that 
\[
\mathbf{V}_{km}=\Delta t\phi_{k}(t_{m}),\quad\mathbf{V}'_{km}=\Delta t\phi_{k}'(t_{m})
\]
\item Compute Gram matrix $\mathbf{G}=\mathbf{V}\Theta(\mathbf{y})$ and
right-hand side $\mathbf{b}=-\mathbf{V}'\mathbf{y}$ so that $\mathbf{G}_{kj}=~\langle\phi_{k},f_{j}(\mathbf{y})\rangle$
and $\mathbf{b}_{kd}=-\langle\phi_{k}',\mathbf{y}_{d}\rangle$ 
\item Solve the generalized least-squares problem with $\ell_{2}$-regularization
\[
\widehat{\mathbf{w}}=\text{argmin}_{\mathbf{w}}\left\{ (\mathbf{G}\mathbf{w}-\mathbf{b})^{T}\Sigma^{-1}(\mathbf{G}\mathbf{w}-\mathbf{b})+\gamma^{2}\left\Vert \mathbf{w}\right\Vert _{2}^{2}\right\} ,
\]
using sequential thresholding with parameter $\lambda$ to enforce
sparsity.
\end{enumerate}
With this as our core algorithm, we can now consider a residual analysis
(Section \ref{subsec:Residual-Analysis}) leading to a weighted least
squares solution. We can also develop theoretical results related
to the test functions (Section \ref{subsec:Test-Function-Basis}),
yielding a more thorough understanding of the impact of using uniform
(Section \ref{subsec:Strategy-1:-Uniform}) and adaptive (Section
\ref{subsec:Strategy-2:-Adaptive}) placement of test functions along
the time axis.

\subsection{\label{subsec:Residual-Analysis}Residual Analysis}

Performance of WSINDy is determined by the behavior of the residuals
\[
\mathcal{R}(\mathbf{w};\phi_{k}):=\left(\mathbf{G}\mathbf{w}-\mathbf{b}\right)_{k}\ \in\ \mathbb{R}^{1\times D},
\]
denoted $\mathcal{R}(\mathbf{w})\in\mathbb{R}^{K\times D}$ for the
entire residual matrix. Here we analyze the residual to highlight
key aspects for future analysis, as well as to arrive at an appropriate
choice of approximate covariance $\Sigma$. We also provide a heurisic
argument in favor of placing test functions near steep gradients in
the dynamics.\\

A key difficulty in recovering the true weights $\mathbf{w}^{\star}$
is that in general, $\mathbb{E}[\mathcal{R}(\mathbf{w}^{\star})]\neq0$
due to nonlinearities present in $\Theta(\mathbf{y})$, thus the recovered
weights $\widehat{\mathbf{w}}$ will be inherently biased. Nevertheless,
we can isolate the dominant error terms by expanding out the residual
and linearizing around the true trajectory $\mathbf{x}(t)$: 
\begin{align*}
\mathcal{R}(\mathbf{w};\phi_{k}) & =\left\langle \phi_{k},\ \Theta(\mathbf{y})\mathbf{w}\right\rangle +\left\langle \phi_{k}',\ \mathbf{y}\right\rangle \\
 & =\left\langle \phi_{k},\ \Theta(\mathbf{y})(\mathbf{w}-\mathbf{w}^{\star})\right\rangle +\left\langle \phi_{k},\ \Theta(\mathbf{y})\mathbf{w}^{\star}\right\rangle +\left\langle \phi_{k}',\ \mathbf{y}\right\rangle \\
 & =\underbrace{\left\langle \phi_{k},\ \Theta(\mathbf{y})(\mathbf{w}-\mathbf{w}^{\star})\right\rangle }_{R_{1}}+\underbrace{\left\langle \phi_{k},\ \epsilon\nabla\mathbf{F}(\mathbf{x})\right\rangle }_{R_{2}}+\underbrace{\left\langle \phi_{k}',\ \epsilon\right\rangle }_{R_{3}}+I_{k}+\mathcal{O}(\epsilon^{2})
\end{align*}
where $\nabla F(\mathbf{x})_{d'd}=\frac{\partial\mathbf{F}_{d}}{\partial\mathbf{x}_{d'}}(\mathbf{x})$.
The errors manifest in the following ways: 
\begin{itemize}
\item $R_{1}$ is the misfit between $\mathbf{w}$ and $\mathbf{w}^{\star}$ 
\item $R_{2}$ results from measurement error in trial gridfunctions $f_{j}(\mathbf{y})=f_{j}(\mathbf{x}+\epsilon)$ 
\item $R_{3}$ results from replacing $\mathbf{x}$ with $\mathbf{y}=\mathbf{x}+\epsilon$
in the left-hand side of \eqref{phiI} 
\item $I_{k}$ is the integration error 
\item $\mathcal{O}(\epsilon^{2})$ is the remainder term in the truncated
Taylor expansion of $\mathbf{F}(\mathbf{y})$ around $\mathbf{x}$:
\[
\mathbf{F}(\mathbf{y}_{m})=\mathbf{F}(\mathbf{x}(t_{m}))+\epsilon_{m}\nabla\mathbf{F}(\mathbf{x}(t_{m}))+\mathcal{O}(|\epsilon_{m}|^{2}).
\]
\end{itemize}
Clearly, recovery of $\mathbf{F}$ when $\epsilon=0$ is straight
forward: $R_{1}$ and $I_{k}$ are the only error terms, thus one
only needs to select a quadrature scheme that ensures that the integration
error $I_{k}$ is negligible and $\widehat{\mathbf{w}}=\mathbf{w}^{\star}$
will be the minimizer. A primary focus of this study is the use of
a specific family of piecewise polynomial test functions $\mathcal{S}$
defined below for which the trapezoidal rule is highly accurate (see
Lemma \ref{traplemm}). Figure \ref{oz_fig} demonstrates this fact
on noise-free data.\\

For $\epsilon>0$, accurate recovery of $\mathbf{F}$ requires one
to choose hyperparameters that exemplify the true misfit term $R_{1}$
by enforcing that the other error terms are of lower order. We look
for $(\phi_{k})_{k\in[K]}$ and $\Sigma=\textbf{C}\textbf{C}^{T}$
that approximately enforce $\textbf{C}^{-1}\mathcal{R}(\mathbf{w}^{\star})\sim\mathcal{N}(0,\sigma^{2}\mathbf{I})$,
justifying the least squares approach. In the next subsection we address
the issue of approximating the covariance matrix, providing justification
for using $\Sigma=\mathbf{V}'(\mathbf{V}')^{T}$. The following subsection
provides a heuristic argument for how to reduce corruption from the
error terms $R_{2}$ and $R_{3}$ by placing test functions near steep
gradients in the data.

\subsubsection{\textbf{Approximate Covariance} $\Sigma$}

Neglecting $I_{k}$ and $\mathcal{O}(\epsilon^{2})$, we can rewrite
$\mathcal{R}(\mathbf{w};\phi_{k})$ with $R_{2}$ and $R_{3}$ together
as 
\[
\mathcal{R}(\mathbf{w};\phi_{k})=R_{1}+\mathbf{Z}_{k}=R_{1}+\sum_{m=1}^{M}\epsilon_{m}\mathbf{T}_{m}\Delta t
\]
where 
\[
\mathbf{T}_{m}=\phi'_{k}(t_{m})\mathbf{I}_{D}+\phi_{k}(t_{m})\nabla\mathbf{F}(\mathbf{x}(t_{m}))
\]
is the linearized operator left-multiplying the noise vector
\[
\epsilon_{m}=(\epsilon_{1m}\ \epsilon_{2m}\ \dots\ \epsilon_{Dm})
\]
at timestep $t_{m}$ and $\mathbf{I}_{D}$ is the identity in $\mathbb{R}^{D\times D}$.
The true distribution of $\mathbf{Z}$ therefore depends on $\nabla\mathbf{F}$,
which is not known \textit{a priori}. For a leading order approximation
we propose using $\mathbf{T}_{m}\approx\phi'_{k}(t_{m})\mathbf{I}_{D}$
which holds if $\left\Vert \phi'_{k}\right\Vert _{\infty}\gg\left\Vert \phi_{k}\right\Vert _{\infty}$.
We then get that the columns of $\mathbf{Z}$ (corresponding to errors
in each component $\mathbf{F}_{d}$ along the time series) are approximately
i.i.d normal with mean zero and covariance $\mathbf{V}'(\mathbf{V}')^{T}$.
For this reason, we adopt the heuristic $\Sigma=\mathbf{V}'(\mathbf{V}')^{T}$,
or $\Sigma_{ij}=\Delta t^{2}\sum_{m=1}^{M}\phi_{i}'(t_{m})\phi_{j}'(t_{m})=\Delta t\left\langle \phi_{i}',\phi_{j}'\right\rangle $.

\subsubsection{\textbf{Adaptive Refinement}}

Next we show that by localizing $\phi_{k}$ around large $|\dot{\mathbf{x}}|$,
we get an approximate cancellation of the error terms $R_{2}$ and
$R_{3}$. Consider the one-dimensional case ($D=1$) where $m$ is
an arbitrary time index and $\mathbf{y}_{m}=\mathbf{x}(t_{m})+\epsilon$
is an observation. When $|\dot{\mathbf{x}}(t_{m})|$ is large compared
to $\epsilon$, we approximately have 
\begin{equation}
\mathbf{y}_{m}=\mathbf{x}(t_{m})+\epsilon_{m}\approx\mathbf{x}(t_{m}+\delta t)\approx\mathbf{x}(t_{m})+\delta t\mathbf{F}(\mathbf{x}(t_{m}))\label{adaptarg}
\end{equation}
for some small $\delta t$, i.e.\ the perturbed value $\mathbf{y}_{m}$
lands close to the true trajectory $\mathbf{x}$ at the time $t_{m}+\delta t$.
To understand the heuristic behind this approximation, let $\delta t$
be the point of intersection between the tangent line to $\mathbf{x}(t)$
at $t_{m}$ and $\mathbf{x}(t_{m})+\epsilon$. Then 
\[
\delta t=\frac{\epsilon}{\dot{\mathbf{x}}(t_{m})},
\]
hence $|\dot{\mathbf{x}}(t_{m})|\gg\epsilon$ implies that $\mathbf{x}(t_{m})+\epsilon$
will approximately lie on the true trajectory. As well, regions where
$|\dot{\mathbf{x}}(t_{m})|$ is small will not yield accurate recovery
in the case of noisy data, since perturbations are more likely to
exit the relevant region of phase space. If we linearize $\mathbf{F}$
using the approximation \eqref{adaptarg} we get 
\begin{equation}
\mathbf{F}(\mathbf{y}_{m})\approx\mathbf{F}(\mathbf{x}(t_{m}))+\delta t\mathbf{F}'(\mathbf{x}(t_{m}))\mathbf{F}(\mathbf{x}(t_{m}))=\mathbf{F}(\mathbf{x}(t_{m}))+\delta t\ddot{\mathbf{x}}(t_{m}).\label{adaptarg2}
\end{equation}
Assuming $\phi_{k}$ is sufficiently localized around $t_{m}$, \eqref{adaptarg}
also implies that 
\[
\left\langle \phi_{k}',\mathbf{x}\right\rangle +\underbrace{\left\langle \phi_{k}',\epsilon\right\rangle }_{R_{3}}=\left\langle \phi_{k}',\mathbf{y}\right\rangle \approx\left\langle \phi_{k}',\mathbf{x}\right\rangle +\delta t\left\langle \phi_{k}',\mathbf{F}(\mathbf{x})\right\rangle ,
\]
hence $R_{3}\approx\delta t\left\langle \phi_{k}',\mathbf{F}(\mathbf{x})\right\rangle $,
while \eqref{adaptarg2} implies 
\begin{align*}
\left\langle \phi_{k},\Theta(\mathbf{y})\mathbf{w}\right\rangle  & =\underbrace{\left\langle \phi_{k},\Theta(\mathbf{y})(\mathbf{w}-\mathbf{w}^{\star})\right\rangle }_{=R_{2}}+\left\langle \phi_{k},\mathbf{F}(\mathbf{y})\right\rangle \\
 & \approx\left\langle \phi_{k},\Theta(\mathbf{y})(\mathbf{w}-\mathbf{w}^{\star})\right\rangle +\left\langle \phi_{k},\mathbf{F}(\mathbf{x})\right\rangle +\underbrace{\delta t\left\langle \phi_{k},\ddot{\mathbf{x}}\right\rangle }_{\approx R_{2}}\\
 & =\left\langle \phi_{k},\Theta(\mathbf{y})(\mathbf{w}-\mathbf{w}^{\star})\right\rangle +\left\langle \phi_{k},\mathbf{F}(\mathbf{x})\right\rangle -\delta t\left\langle \phi_{k}',\mathbf{F}(\mathbf{x})\right\rangle 
\end{align*}
having integrated by parts. Collecing the terms together yields that
the residual takes the form 
\[
\mathcal{R}(\mathbf{w};\phi_{k})=\left\langle \phi_{k}',\mathbf{y}\right\rangle +\left\langle \phi_{k},\Theta(\mathbf{y})\mathbf{w}\right\rangle \approx R_{2},
\]
and we see that $R_{2}$ and $R_{3}$ have effectively canceled. In
higher dimensions this interpretation does not appear to be as illuminating,
but nevertheless, for any given coordinate $\mathbf{x}_{d}$, it does
hold that terms in the error expansion vanish around points $t_{m}$
where $|\dot{\mathbf{x}}_{d}|$ is large, precisely because $\mathbf{x}_{d}(t_{m})+\epsilon\approx\mathbf{x}_{d}(t_{m}+\delta t)$.

\subsection{\label{subsec:Test-Function-Basis}Test Function Basis $(\phi_{k})_{k\in[K]}$}

Here we introduce a test function space $\mathcal{S}$ and quadrature
scheme to minimize integration errors and enact the heuristic arguments
above, which rely on $\phi_{k}$ sufficiently localized and satisfying
$|\phi_{k}'|\gg|\phi_{k}|$. We define the space $\mathcal{S}$ of
unimodal piece-wise polynomials of the form 
\begin{equation}
\phi(t)=\begin{cases}
C(t-a)^{p}(b-t)^{q} & t\in(a,b),\\
0 & \text{otherwise},
\end{cases}\label{testfcn}
\end{equation}
where $(a,b)\subset\mathbf{t}\times\mathbf{t}$ satisfies $a<b$ and
$p,q\geq1$. The normalization 
\[
C=\frac{1}{p^{p}q^{q}}\left(\frac{p+q}{b-a}\right)^{p+q}
\]
ensures that $\left\Vert \phi\right\Vert _{\infty}=1$. Functions
$\phi\in\mathcal{S}$ are non-negative, unimodal, and compactly supported
in $[0,T]$ with $\lfloor\min\{p,q\}\rfloor-1$ continuous derivatives.
Larger $p$ and $q$ imply faster decay towards the endpoints of the
support.\\

To ensure the integration error in approximating inner products $\langle f_{j},\phi_{k}\rangle$
is negligible, we rely on the following lemma, which provides a bound
on the error in discretizing the weak derivative relation 
\begin{equation}
-\int\phi'f\,dt=\int\phi f'\,dt\label{weakderiv}
\end{equation}
using the trapezoidal rule for compactly supported $\phi$. Following
the lemma we introduce two strategies for choosing the parameters
of the test functions $(\phi_{k})_{k\in[K]}\subset\mathcal{S}$.
\begin{lem}
[Numerical Error in Weak Derivatives]\label{traplemm} Let $f,\phi$
have continuous derivatives of order $p$ and define $t_{j}=a+j\frac{b-a}{N}=a+j\Delta t$.
Then if $\phi$ has roots $\phi(a)=\phi(b)=0$ of multiplicity $p$,
then 
\begin{equation}
\frac{\Delta t}{2}\sum_{j=0}^{N-1}\big[g(t_{j})+g(t_{j+1})\big]=\mathcal{O}(\Delta t^{p+1}),\label{traplemmscale}
\end{equation}
where $g(t)=\phi'(t)f(t)+\phi(t)f'(t)$. In other words, the composite
trapezoidal rule discretizes the weak derivative relation \eqref{weakderiv}
to order $p+1$.
\end{lem}

\begin{proof}
This is a simple consequence of the Euler-Maclaurin formula. If $g:[a,b]\to\mathbb{C}$
is a smooth function, then 
\[
\frac{\Delta t}{2}\sum_{j=0}^{N-1}\left[g(t_{j})+g(t_{j+1})\right]\sim\int_{a}^{b}g(t)\,dt+\sum_{k=1}^{\infty}\,\frac{\Delta t^{2k}B_{2k}}{(2k)!}\left(g^{(2k-1)}(b)-g^{(2k-1)}(a)\right),
\]
where $B_{2k}$ are the Bernoulli numbers. The asymptotic expansion
provides corrections to the trapezoidal rule that realize machine
precision accuracy up until a certain value of $k$, after which terms
in the expansion grow and the series diverges \cite[Ch. ~3]{dahlquist2008numerical}.
In our case, $g(t)=\phi'(t)f(t)+\phi(t)f'(t)$ where the root conditions
on $\phi$ imply that 
\[
\int_{a}^{b}g(t)\,dt=0\quad\text{and}\quad g^{(k)}(b)=g^{(k)}(a)=0,\quad0\leq k\leq p-1.
\]
So for $p$ odd, we have that 
\begin{align*}
\frac{\Delta t}{2}\sum_{j=0}^{N-1}\left[g(t_{j})+g(t_{j+1})\right] & \sim\sum_{k=(p+1)/2}^{\infty}\,\frac{\Delta t^{2k}B_{2k}}{(2k)!}\left(g^{(2k-1)}(b)-g^{(2k-1)}(a)\right)\\
 & =\frac{B_{p+1}}{(p+1)!}(\phi^{(p)}(b)f(b)-\phi^{(p)}(a)f(a))\Delta t^{p+1}+\mathcal{O}\left(\Delta t^{p+2}\right).
\end{align*}
For even $p$, the leading term is $\mathcal{O}(\Delta t^{p+2})$
with a slightly different coefficient. 
\end{proof}
For $\phi\in\mathcal{S}$ with $p=q$, the exact leading order error
in term in \eqref{traplemmscale} is 
\[
\frac{2^{p}\,B_{p+1}}{p+1}\,\big(f(b)-f(a)\big)\Delta t^{p+1},
\]
which is negligible for a wide range of reasonable $p$ and $\Delta t$
values. The Bernoulli numbers eventually start growing like $p^{p}$,
but for smaller values of $p$ they are moderate. For instance, with
$\Delta t=0.1$ and $f(b)-f(a)=1$, this error term is $o(1)$ up
until $p=85$, where it takes the value $0.495352$, while for $\Delta t=0.01$,
the error is below machine precision for all $p$ between 7 and 819.
For these reasons, in what follows we choose test functions $(\phi_{k})_{k\in[K]}\subset\mathcal{S}$
and discretize all integrals using the trapezoidal rule. Unless otherwise
stated, each function $\phi_{k}$ satsifies $p=q$ and so is fully
determined by the tuple $\{p_{k},a_{k},b_{k}\}$ indicating its polynomial
degree and support. In the next two subsections we propose two different
strategies for determining $\phi_{k}$ using the data $\mathbf{y}$.

\subsubsection{\textbf{\label{subsec:Strategy-1:-Uniform}Strategy 1: Uniform Grid}}

The simplest strategy for choosing a test function basis $(\phi_{k})_{k\in[K]}\subset\mathcal{S}$
is to place $\phi_{k}$ uniformly on the interval $[0,T]$ with fixed
degree $p$ and support size $L$ denoting the number of timepoints
in $\mathbf{t}$ that each $\phi_{k}$ is supported on. Specifically
we propose the following three steps to automate this process, requiring
only that the user specify the polynomial degree $p$ and the total
number of basis functions $K$, through the values of two hyperparameters
$\rho$ and $s$ that relate to the residual analysis above.

\noindent \textit{Step 1: Choosing $L$:} We propose to fix the support
size of each $\phi_{k}$ to be 
\begin{equation}
L=\frac{1}{2}\left(\frac{M}{\text{argmax}_{n}\left\vert \mathbf{Y}\right\vert _{n}}\right)\label{UG}
\end{equation}
where $\left\vert \mathbf{Y}\right\vert _{n}$ is the magnitude of
the $n$th Fourier mode of $\mathbf{y}$ minus its mean. In this way
the action of the test functions exploits that large changes in the
dynamics are most likely to occur within time intervals of length
equal to the largest Fourier mode.\\

\noindent \textit{Step 2: Determining $p$:} In light of the derivation
above of the approximate covariance matrix $\Sigma$, we define $\rho:=\left\Vert \phi'_{k}\right\Vert _{\infty}/\left\Vert \phi_{k}\right\Vert _{\infty}$,
where larger $\rho$ indicates better agreement with $\Sigma$. The
value of $\rho$ selects the polynomial degree $p$ as follows: analytically,
\[
\rho=\frac{2\sqrt{2p-1}}{b-a}\left(\frac{1-1/p}{1-1/2p}\right)^{p-1}=\frac{1.67\dots}{b-a}\left(p^{1/2}+p^{-1/2}\right)+o(p^{-1/2}).
\]
As a leading order approximate we set $p=\lceil\rho^{2}\frac{|b-a|^{2}}{2.8}\rceil$.

\noindent \textit{Step 3: Determining $K$:} Next we introduce the
shift parameter $s\in[0,1]$ defined by 
\[
s=\phi_{k}(t^{*})=\phi_{k+1}(t^{*}),
\]
which determines $K$ from $p$ and $L$. In words, $s$ is the height
of intersection between $\phi_{k}$ and $\phi_{k+1}$ and measures
the amount of overlap between successive test functions, which factors
into the covariance matrix $\Sigma$. This fixes $a_{k+1}-a_{k}=\Delta tL\sqrt{1-s^{1/p}}$
for each pair of neighboring test functions. Larger $s$ implies that
neighboring basis functions overlap on more points, with $s=1$ indicating
that $\phi_{k}=\phi_{k+1}$. Specifically, neighboring basis functions
overlap on $\lfloor L(1-\sqrt{1-s^{1/p}})\rfloor$ timepoints.

In Figures \ref{lownz_duff} and \ref{lownz_vp} we vary the parameters
$\rho$ and $s$ and observe that results agree with intuition: larger
$\rho$ and larger $s$ lead to better recovery of $\mathbf{w}^{\star}$.

\noindent $\widehat{\mathbf{w}}=\textbf{WSINDy\_UG}\left(\mathbf{y},\mathbf{t}\ ;\ (f_{j})_{j\in[J]},\rho,s,\lambda,\gamma\right)$: 
\begin{enumerate}
\item Construct matrix of trial gridfunctions $\Theta(\mathbf{y})=\big[f_{1}(\mathbf{t},\mathbf{y})\,|\,\dots\,|\,f_{J}(\mathbf{t},\mathbf{y})\big]$ 
\item Construct integration matrices $\mathbf{V}$, $\mathbf{V}'$ such
that 
\[
\mathbf{V}_{km}=\Delta t\phi_{k}(t_{m}),\quad\mathbf{V}'_{km}=\Delta t\phi_{k}'(t_{m})
\]
with the test functions $(\phi_{k})_{k\in[K]}$ determined by $\rho,s$
as described above 
\item Compute Gram matrix $\mathbf{G}=\mathbf{V}\Theta(\mathbf{y})$ and
right-hand side $\mathbf{b}=-\mathbf{V}'\mathbf{y}$ so that $\mathbf{G}_{kj}=~\langle\phi_{k},f_{j}(\mathbf{y})\rangle$
and $\mathbf{b}_{kd}=-\langle\phi_{k}',\mathbf{y}_{d}\rangle$ 
\item Compute approximate covariance and Cholesky factorization $\Sigma=\mathbf{V}'(\mathbf{V}')^{T}=\textbf{C}\textbf{C}^{T}$ 
\item Solve the generalized least-squares problem with $\ell_{2}$-regularization
\[
\widehat{\mathbf{w}}=\text{argmin}_{\mathbf{w}}\left\{ (\mathbf{G}\mathbf{w}-\mathbf{b})^{T}\Sigma^{-1}(\mathbf{G}\mathbf{w}-\mathbf{b})+\gamma^{2}\left\Vert \mathbf{w}\right\Vert _{2}^{2}\right\} ,
\]
using sequential thresholding with parameter $\lambda$ to enforce
sparsity. 
\end{enumerate}

\subsubsection{\textbf{\label{subsec:Strategy-2:-Adaptive}Strategy 2: Adaptive
Grid}}

Motivated by the arguments above, we now introduce an algorithm for
constructing a test function basis localized near points of large
change in the dynamics. This occurs in three steps: 1) construct a
weak approximation to the derivative of the dynamics $\mathbf{v}\approx\dot{\mathbf{x}}$,
2) sample $K$ points $\textbf{c}$ from a cumulative distribution
$\psi$ with density proportional to the total variation $|\mathbf{v}|$,
3) construct test functions centered at $\textbf{c}$ using a width-at-half-max
parameter $r_{whm}$ to determine the parameters $(p_{k},a_{k},b_{k})$
of each basis function $\phi_{k}$. Each of these steps is numerically
stable and carried out independently along each coordinate of the
dynamics. A visual diagram is provided in Figure \ref{adapt}.\\

\begin{figure}
\begin{tabular}{cc}
\includegraphics[clip,width=0.45\textwidth]{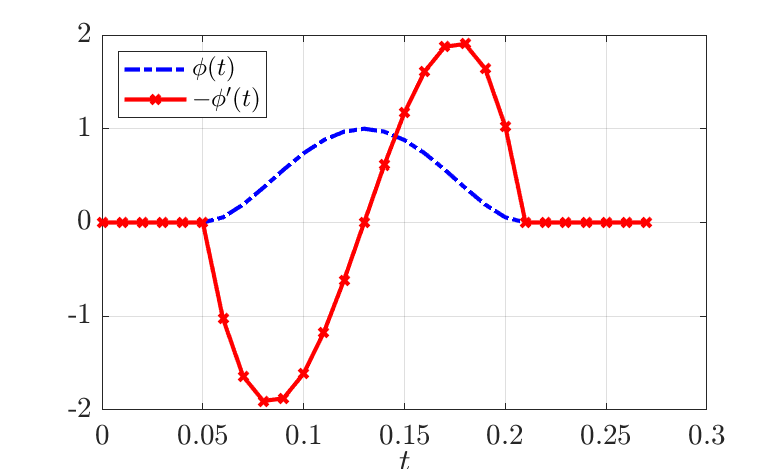}  & \includegraphics[clip,width=0.45\textwidth]{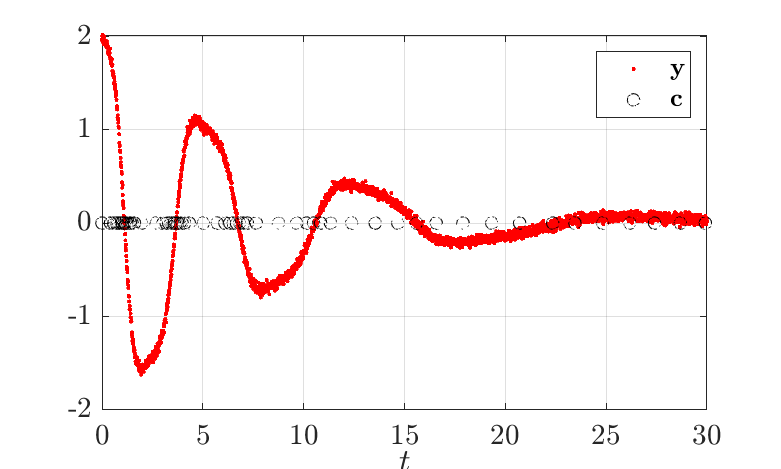} \tabularnewline
\includegraphics[clip,width=0.45\textwidth]{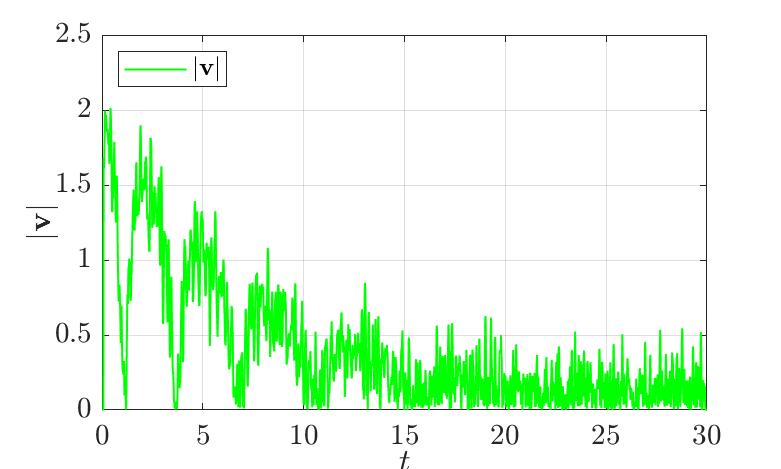}  & \includegraphics[clip,width=0.45\textwidth]{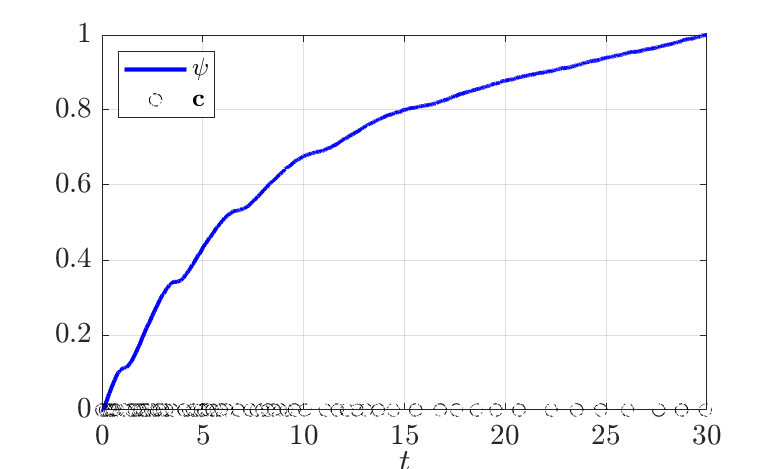} \tabularnewline
\end{tabular}\caption{Counter-clockwise from top left: test function $\phi$ and derivative
$-\phi'$ used to compute $\mathbf{v}$, approximate total variation
$|\mathbf{v}|$, cumulative distribution $\psi=\int^{t}|\mathbf{v}|\,dt$,
noisy data $\mathbf{y}$ from the Duffing equation and resulting test
functons centers $\mathbf{c}$. As desired, the centers $\mathbf{c}$
are clustered near steep gradients in the dynamics despite measurement
noise.}
\label{adapt} 
\end{figure}

\noindent \textit{Step 1: Weak Derivative Approximation:} Define $\mathbf{v}:=-\textbf{V}_{p}'\,\mathbf{y}$,
where the matrix $-\textbf{V}_{p}'$ enacts a linear convolution with
the derivative of a piece-wise polynomial test function $\phi\in\mathcal{S}$
of degree $p$ and support size $s$ so that 
\[
\mathbf{v}_{m}=-\left\langle \phi',\mathbf{y}\right\rangle =\left\langle \phi,\dot{\mathbf{y}}\right\rangle \approx\dot{\mathbf{y}}_{m}.
\]
The parameters $s$ and $p$ are chosen by the user, with $s=4$ and
$p\geq2$ corresponding to taking a centered finite difference derivative
with 3-point stencil. Larger $s$ results in more smoothing and minimizes
the corruption from noise while still capturing the correct large
deviations in the dynamics. For all examples we let $p=2$ and $s=16$
and note that greater disparity between $p$ and $s$ results in more
pronounced localization (less uniform distribution) of test functions.
\\

\noindent \textit{Step 2: Selecting $\textbf{c}$:} Having computed
$\mathbf{v}$, define $\psi$ to be the cumulative sum of $|\mathbf{v}|$
normalized so that $\max{\psi}=1$. In this way $\psi$ is a valid
cumulative distribution function with density proportional to the
total variation of $\mathbf{y}$. We then find $\textbf{c}$ by sampling
form $\psi$. Let $U=[0,\frac{1}{K},\frac{2}{K},\dots,\frac{K-1}{K}]$
with $K$ being the number of the test functions, we then define $\textbf{c}=\psi^{-1}(U)$,
or numerically, 
\[
c_{k}=\min\{t\in\mathbf{t}\ :\ \psi(t)\geq U_{k}\}.
\]
This stage requires the user to select the number of test functions
$K$.\\

\noindent \textit{Step 3: Construction of Test functions $(\phi_{k})_{k\in[K]}$:}
Having chosen the location $c_{k}$ of the centerpoint for each test
function $\phi_{k}$, we are left to choose the degree $p_{k}$ of
the polynomial and the supports $[a_{k},\,b_{k}]$. The degree is
chosen according to the width-at-half-max parameter $r_{whm}$, which
specifies the difference in timepoints between each center $c_{k}$
and $\text{arg}_{t}\{\phi_{k}(t)=1/2\}$, while the supports are chosen
such that $\phi_{k}(b_{k}-\Delta t)=10^{-16}$. This gives us a nonlinear
system of two equations in two unknowns which can be easily solved
(i.e. using MATLAB's \texttt{fzero}). This can be done for one reference
test functions and the rest of the weights obtained by translation.\\

The adaptive grid Weak SINDy algorithm is summarized as follows:\\

\noindent $\widehat{\mathbf{w}}=\textbf{WSINDy\_AG}\left(\mathbf{y},\mathbf{t}\ ;\ (f_{j})_{j\in[J]},p,s,K,r_{whm},\lambda,\gamma\right)$: 
\begin{enumerate}
\item Construct matrix of trial gridfunctions $\Theta(\mathbf{y})=\big[f_{1}(\mathbf{t},\mathbf{y})\,|\,\dots\,|\,f_{J}(\mathbf{t},\mathbf{y})\big]$ 
\item Construct integration matrices $\mathbf{V}$, $\mathbf{V}'$ such
that 
\[
\mathbf{V}_{km}=\Delta t\phi_{k}(t_{m}),\quad\mathbf{V}'_{km}=\Delta t\phi_{k}'(t_{m}),
\]
with test functions $(\phi_{k})_{k\in[K]}$ determined by $p,s,K,r_{whm}$
as described above 
\item Compute Gram matrix $\mathbf{G}=\mathbf{V}\Theta(\mathbf{y})$ and
right-hand side $\mathbf{b}=-\mathbf{V}'\mathbf{y}$ so that $\mathbf{G}_{kj}=~\langle\phi_{k},f_{j}(\mathbf{y})\rangle$
and $\mathbf{b}_{kd}=-\langle\phi_{k}',\mathbf{y}_{d}\rangle$ 
\item Compute approximate covariance and Cholesky factorization $\Sigma=\mathbf{V}'(\mathbf{V}')^{T}=\textbf{C}\textbf{C}^{T}$ 
\item Solve the generalized least-squares problem with $\ell_{2}$-regularization
\[
\widehat{\mathbf{w}}=\text{argmin}_{\mathbf{w}}\left\{ (\mathbf{G}\mathbf{w}-\mathbf{b})^{T}\Sigma^{-1}(\mathbf{G}\mathbf{w}-\mathbf{b})+\gamma^{2}\left\Vert \mathbf{w}\right\Vert _{2}^{2}\right\} ,
\]
using sequential thresholding with parameter $\lambda$ to enforce
sparsity. 
\end{enumerate}
The parameters $p$ and $s$ play a role in determining how localized
the test function basis is around steep gradients and ultimately depend
on the timestep $\Delta t$. As mentioned above, we set $p=2$ and
$s=16$ throughout as this produces sufficient localization for the
examples featured in this article. For simplicity we fix the number
of test functions $K$ to be a multiple of the number of trial functions
(i.e. $K=J,\,2J,\,3J$ etc.). For larger noise it is necessary to
use a larger basis, while for small noise it as often sufficient to
set $K=J$. The optimal value of $r_{whm}$ depends on the timescales
of the dynamics and can be chosen from the data using the Fourier
transform as in the uniform grid case, however for simplicity we set
$r_{whm}=30$ throughout.

\section{Numerical Experiments}

We now show that WSINDy is capable of recovering the correct dynamics
to high accuracy over a range of signal-to-noise ratios. To generate
true trajectory data we use MATLAB's \texttt{ode45} with absolute
and relative tolerance $1$e$-10$ and sampling rate of $\Delta t=0.01,$
unless otherwise specified. We choose a fixed sampling rate $\Delta t$
for simplicity and uniformity across examples but note that a detailed
study of the dependency of the algorithm on $\Delta t$ is not presented
here. While the results in this section hold for a wide range of $\Delta t$
(results not presented here), different sampling rates may result
in different choices of hyperparameters (e.g., $\rho$ and $s$ in
the case of WSINDy\_UG). White Gaussian noise with mean zero and variance
$\sigma^{2}$ is then added to the exact trajectories, where $\sigma$
is computed by specifying the signal-to-noise ratio $\sigma_{SNR}$
and setting 
\[
\sigma=\sigma_{SNR}\left\Vert \mathbf{x}\right\Vert _{RMS}\quad\text{where}\quad\left\Vert \mathbf{x}\right\Vert _{RMS}=\sqrt{\frac{1}{DM}\sum_{d=1}^{D}\sum_{m=1}^{M}|\mathbf{x}_{d}(t_{m})|^{2}}.
\]
In this way, $\left\Vert \epsilon\right\Vert _{RMS}/\left\Vert \mathbf{x}\right\Vert _{RMS}\approx\sigma$.
We examine the following canonical nonlinear systems with variations
in the specified parameters:\\
\vspace{-30pt}

\begin{center}
\begin{align*}
\text{Duffing}\quad & \begin{cases}
\dot{x}_{1}=x_{2},\\
\dot{x}_{2}=-\mu x_{2}-\alpha x_{1}-\beta x_{1}^{3},
\end{cases} &  & \text{variable }\beta, &  & \text{fixed }\mu=0.2,\alpha=0.05,\mathbf{x}(0)=\begin{bmatrix}0\\
2
\end{bmatrix}\\
\text{Van der Pol}\quad & \begin{cases}
\dot{x}_{1}=x_{2},\\
\dot{x}_{2}=\beta x_{2}(1-x_{1}^{2})-x_{1},
\end{cases} &  & \text{variable }\beta, &  & \text{fixed }\mathbf{x}(0)=\begin{bmatrix}0\\
1
\end{bmatrix}\\
\text{Lotka-Volterra}\quad & \begin{cases}
\dot{x}_{1}=\alpha x_{1}-\beta x_{1}x_{2},\\
\dot{x}_{2}=\beta x_{1}x_{2}-2\alpha x_{2},
\end{cases} &  & \text{variable }\beta, &  & \text{fixed }\alpha=1,\mathbf{x}(0)=\begin{bmatrix}1\\
2
\end{bmatrix}\\
\text{Lorenz}\quad & \begin{cases}
\dot{x}_{1}=\sigma(x_{2}-x_{1}),\\
\dot{x}_{2}=x_{1}(\rho-x_{3})-x_{2},\\
\dot{x}_{3}=x_{1}x_{2}-\beta x_{3},
\end{cases} &  & \text{variable }\mathbf{x}(0), &  & \text{fixed }\sigma=10,\beta=8/3,\rho=28
\end{align*}
\par\end{center}

The Duffing equation and Van der Pol oscillator present cases of approximately
linear systems with cubic nonlinearities. Solutions to the Van der
Pol oscillator and Lotka-Volterra system exhibit orbits with variable
speed of motion, in particular regions with rapid change between regions
of very litte variation. For the Lorenz system, we focus on recovering
the system in the chaotic regime. For this reason we fix the parameters
of the differential equation to lie in the region with large Lyupanov
exponents and vary the initial conditions. The initial conditions
are chosen from a uniform distribution, $x_{1},x_{2}\sim U_{[-15,15]}$
and $x_{3}\sim U_{[10,40]}$, which covers the strange attractor.

\subsection{Noise-Free Data}

The goal of the noise-free experiments is to examine the effect of
increasing the polynomial degree $p$ of the test functions to show
that convergence to machine precision is realized in the limit of
large $p$ (i.e. convergence to within the accuracy tolerance of the
ODE solver), as expected from Lemma \ref{traplemm}. Indeed, Figure
\ref{oz_fig} shows that in the zero-noise case ($\epsilon=0$), WSINDy
recovers the correct weight matrix $\mathbf{w}^{\star}$ to within
the tolerance of the ODE solver (fixed at $10^{-10}$) over a wide
range of parameter values for each of the dynamical systems above.
We find that accurate recovery occurs regardless of sparsity enforcement
or regularization, and so we set $\lambda=0.001$ throughout, orders
of magnitude below any of the true coefficients $\mathbf{w}^{\star}$,
and $\gamma=0$. For the data-driven trial basis $(f_{j})_{j\in[J]}$,
we include all polynomials up to degree 5 in the state variables as
well as $\cos(n\mathbf{y}_{d})$, $\sin(n\mathbf{y}_{d})$ for $n=1,2$
and $d\in[D]$. In addition, we find that recovery occurs with the
minimal number of basis functions $K=J$ such that the Gram matrix
$\mathbf{G}=\mathbf{V}\Theta(\mathbf{y})$ is square. We use the uniform
grid approach above with support $L$ selected from the Fourier transform
of $\mathbf{y}$ and shift parameter $s$ fixed to ensure that $K=J$. 

\begin{figure}
\begin{tabular}{cc}
\includegraphics[clip,width=0.5\textwidth]{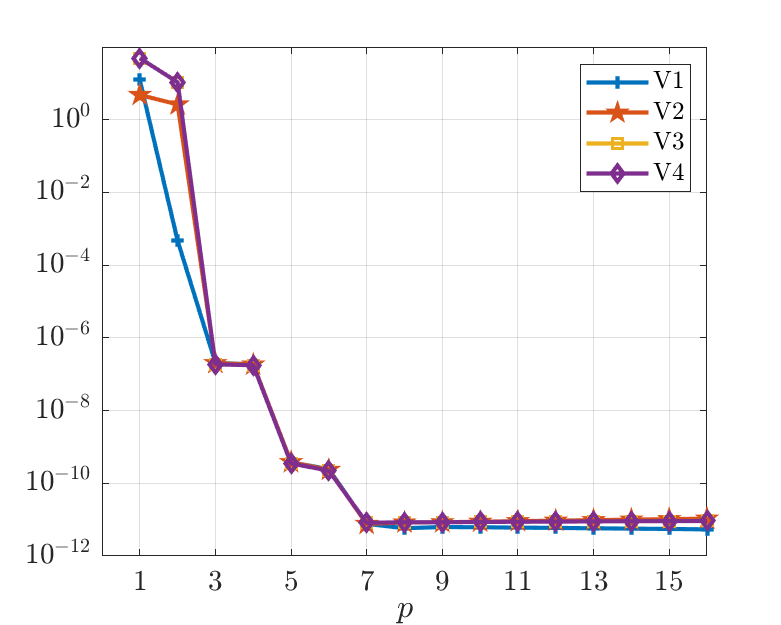}  & \includegraphics[clip,width=0.5\textwidth]{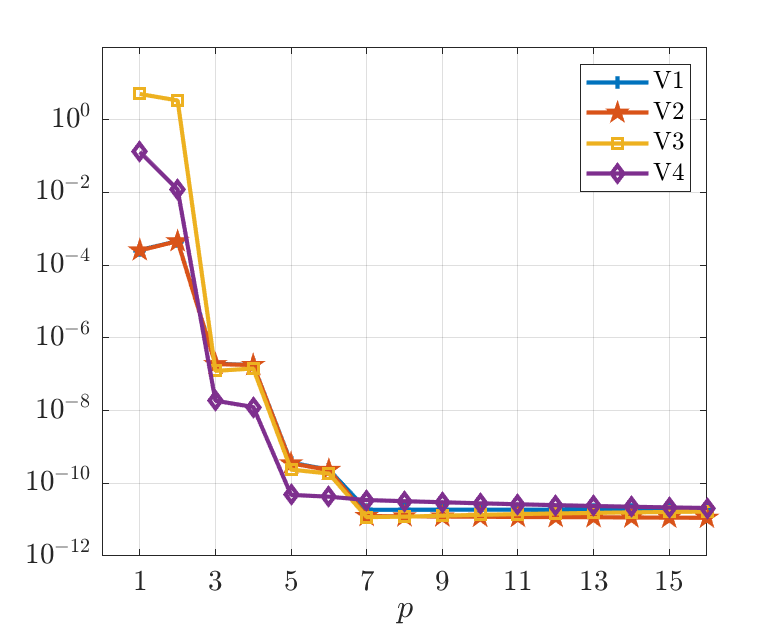} \tabularnewline
Duffing  & Van der Pol\tabularnewline
\includegraphics[clip,width=0.5\textwidth]{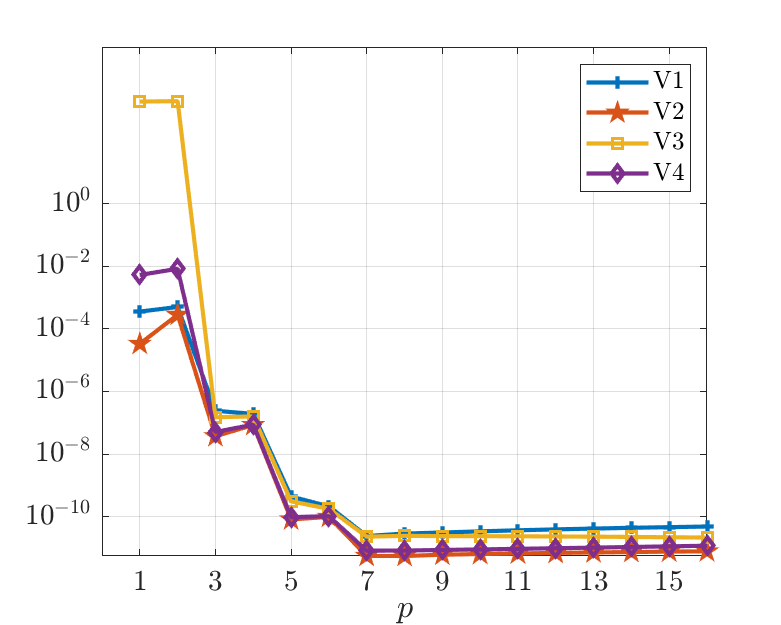}  & \includegraphics[clip,width=0.5\textwidth]{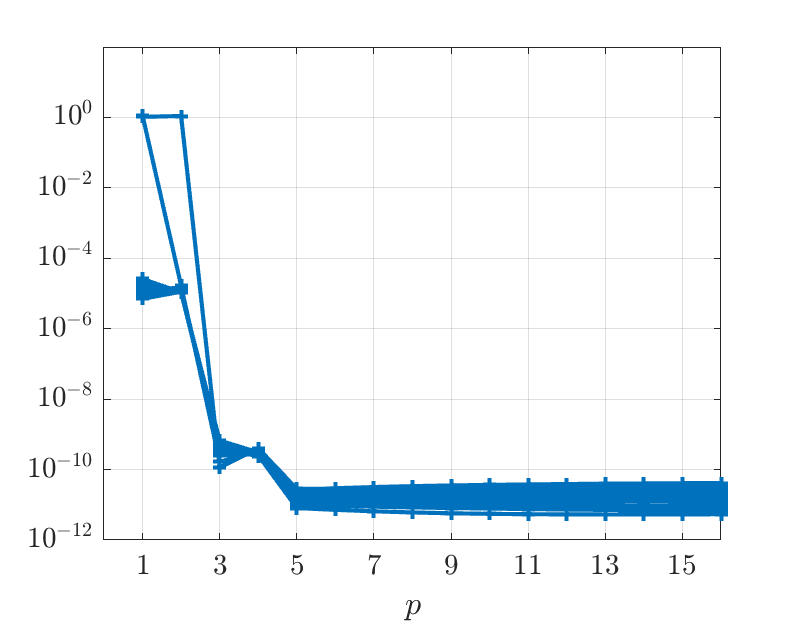} \tabularnewline
Lotka-Volterra  & Lorenz\tabularnewline
 & \tabularnewline
\end{tabular}

~\caption{Noise-free data: plots of relative error $\left\Vert \mathbf{\widehat{w}-\mathbf{w}^{\star}}\right\Vert _{2}/\left\Vert \mathbf{w}^{\star}\right\Vert _{2}$
vs. $p$ when $\epsilon=0$ using WSINDy\_UG. For each system, a range
of parameter values is considered (see Table.\ref{tab:Specifications-for-parameters}).
In each case, the recovered coefficients $\widehat{\mathbf{w}}$ rapidly
converge to within the accuracy of the ODE solver ($10^{-10}$) as
$p$ is increased. For the Duffing equation, Van der Pol oscillator
and Lotka-Volterra system, convergence is shown for parameter sets
V1-V4 spanning several orders of magnitude. For the Lorenz system,
convergence is shown for 40 trajectories, each generated with initial
conditions drawn randomly from a uniform distribution which covers
the strange attractor. }
\label{oz_fig}
\end{figure}
\begin{table}
\begin{centering}
\begin{tabular}{c|ccccc}
 & V1  & V2  & V3  & V4  & Notes \tabularnewline
\hline 
Duffing, $\beta:$ & $0.005$  & $0.08$  & $1$  & $100$  & $\mathbf{t}=0:0.01:30$ \tabularnewline
Van der Pol, $\beta:$ & $0.01$  & $0.1$  & $1$  & $10$  & $\mathbf{t}=0:0.01:30$ \tabularnewline
Lotka-Volterra, $\beta:$  & $0.05$  & $0.1$  & $1$  & $10$  & $\mathbf{t}=0:0.01:30$ \tabularnewline
Lorenz  & -  & -  & -  & -  & $\mathbf{x}(0)\sim U_{[-15,15]^{2}\times[10,40]}$, \tabularnewline
 &  &  &  &  & $\mathbf{t}=(0.001:0.001:10)$\tabularnewline
\end{tabular}
\par\end{centering}
\caption{\label{tab:Specifications-for-parameters}Specifications for parameters
used in illustrating simulations in Figure \ref{oz_fig}.}
\end{table}

\subsection{Small-Noise Regime }

We now turn to the case of low to moderate noise levels, examining
a signal-to-noise ratio $\sigma_{SNR}$ in the range $[10^{-5},0.04]$.
In Figures \ref{lownz_duff} and \ref{lownz_vp} we observe another
nice property of WSINDy, that the error in the coefficients scales
with $\sigma_{SNR}$, in that the recovered coefficients $\widehat{\mathbf{w}}$
have approximately $\log_{10}(10\sigma_{SNR}^{-1})$ significant digits.\\

We again use the uniform grid approach. We examine not only the polynomial
degree $p$ but the number of basis functions $K$ used in recovery.
To reiterate the arguments above, the magnitude of $\phi_{k}'$ compared
to $\phi_{k}$ affects the distribution of the residual, so we define
$\rho:=\left\Vert \phi'_{k}\right\Vert _{\infty}/\left\Vert \phi{}_{k}\right\Vert _{\infty}$
and define the degree $p$ by fixing $\rho$ and then calculating
$p$. In this way, increasing $\rho$ corresponds to increasing $p$.
We look at $\rho\in[1,5]$ which corresponds roughly to $p\in[4,100]$.
This together with the spacing parameter $s$ determines the test
function basis. We enforce that two neighboring basis functions $\phi_{k}$
and $\phi_{k+1}$ intersect at a height of $s$, so that with $s=1$,
the two functions perfectly overlap, and with $s=0$ their supports
are disjoint. In this way, larger $s$ corresponds to higher $K$.
We examine $s\in(0,1)$. For $s=0.5$ (featured in both Figures) we
note that as $\rho$ is varied from $1$ to $5$, the number of basis
functions $K$ ranges from $21$ to $105$, or $K=J$ to $K=5J$.
In each case we set the sparsity and regularization parameters to
$\lambda=\frac{1}{4}\min_{\mathbf{w}^{\star}\neq0}|\mathbf{w}^{\star}|$
and $\gamma=0$ and use a trial basis $(f_{j})_{j\in[J]}$ consisting
of all polynomials up to degree 5 in the state variables.\\

We simulated 200 instantiations of noise for the Duffing equation
and Van der Pol oscillator and for each noisy trajectory examined
a range of the parameter values $s$ and $\rho$. As one might expect
form the noise-free case above, increasing $\rho$ leads monotonically
to better recovery. In addition, increasing $s$ also leads to better
recovery. The mean and standard deviation of the coefficient error
$\left\Vert \mathbf{\widehat{w}-\mathbf{w}^{\star}}\right\Vert _{2}/\left\Vert \mathbf{w}^{\star}\right\Vert _{2}$
are also pictured, along with sample trajectories from the resulting
data-driven dynamical systems, denoted by $\mathbf{x}_{dd}$.

\begin{figure}
\begin{tabular}{c}
\begin{tabular}{cc}
\includegraphics[clip,width=0.45\textwidth]{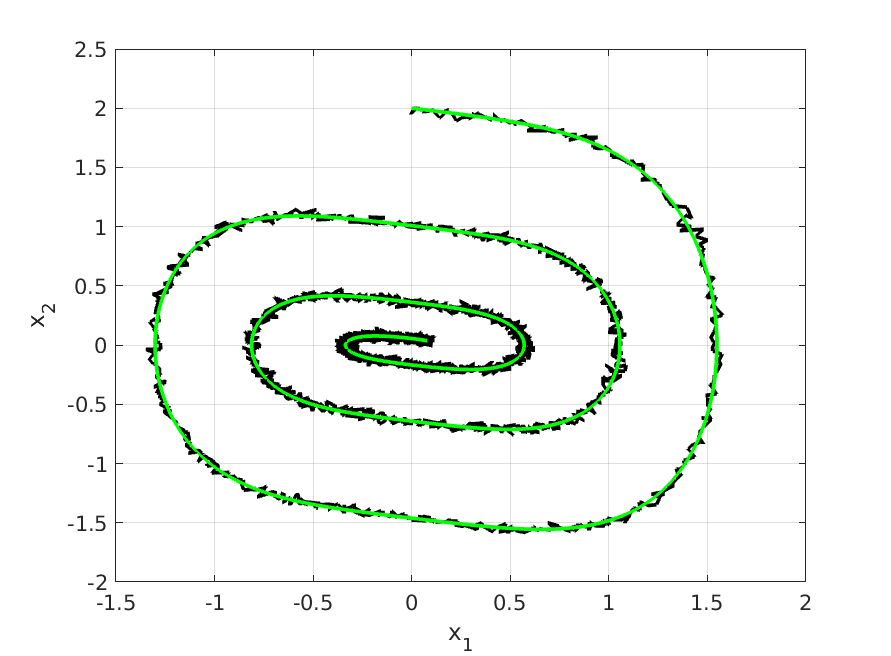}  & \includegraphics[clip,width=0.45\textwidth]{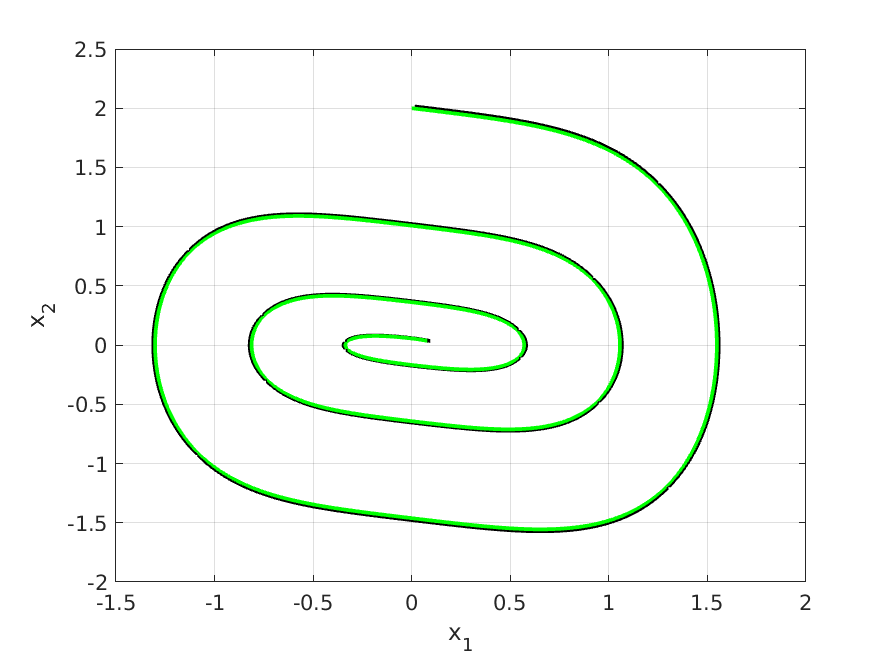} \tabularnewline
\end{tabular}\tabularnewline
\includegraphics[clip,width=1\textwidth]{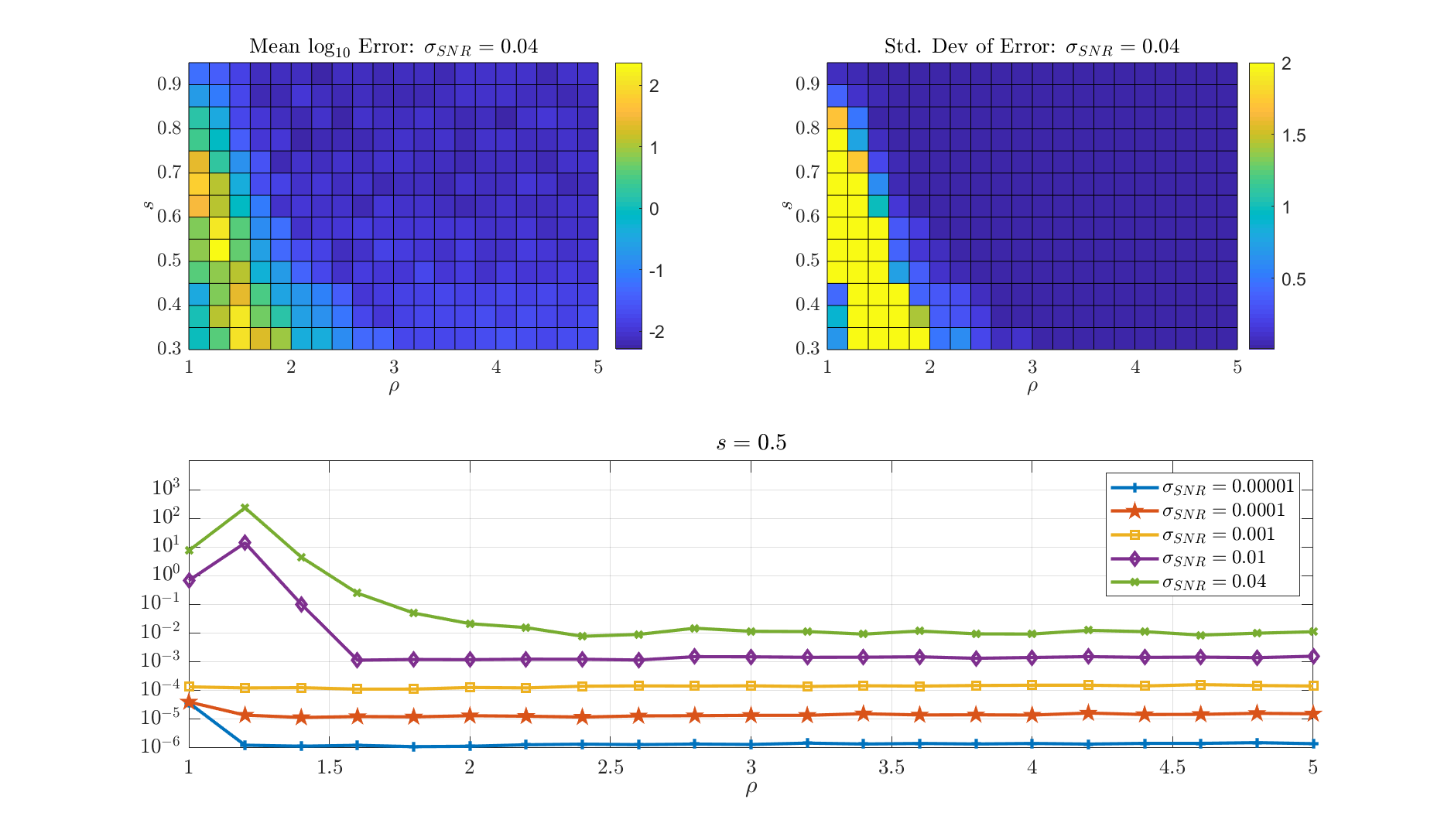} \tabularnewline
\end{tabular}\caption{Low-noise regime: dynamic recovery of the Duffing equation with parameters
$\mu=0.2,\alpha=0.05,\beta=1$. Top row: example trajectory $\mathbf{y}$
(left) and learned dynamics $\mathbf{x}_{dd}$ (right) both plotted
over true data $\mathbf{x}$. Here, $\sigma_{SNR}=0.04$, $\rho=5$,
$s=0.5$ and the coefficient error is $E:=\left\Vert \mathbf{\widehat{w}-\mathbf{w}^{\star}}\right\Vert _{2}/\left\Vert \mathbf{w}^{\star}\right\Vert _{2}=0.0009$.
Middle row: heat map of the $\log_{10}$ average error $E$ (left)
and standard deviation (right) over 200 noisy trajectories with $\sigma_{SNR}=0.04$
with increasing $\rho$ along the $x$-axis and increasing $s$ along
the $y$-axis. Bottom: decreasing error trend for fixed $s=0.5$ for
various $\sigma_{SNR}$. For each $\sigma_{SNR}$ the expected error
falls roughly an order of magnitude below the $\sigma_{SNR}$ as $\rho$
increases.}
\label{lownz_duff} 
\end{figure}

\begin{figure}
\begin{tabular}{c}
\begin{tabular}{cc}
\includegraphics[clip,width=0.45\textwidth]{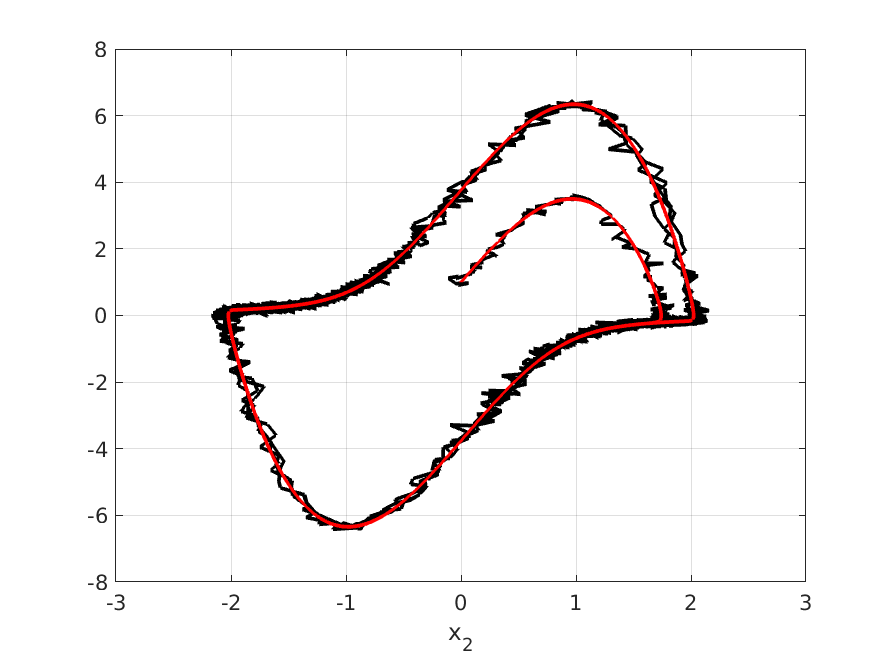}  & \includegraphics[clip,width=0.45\textwidth]{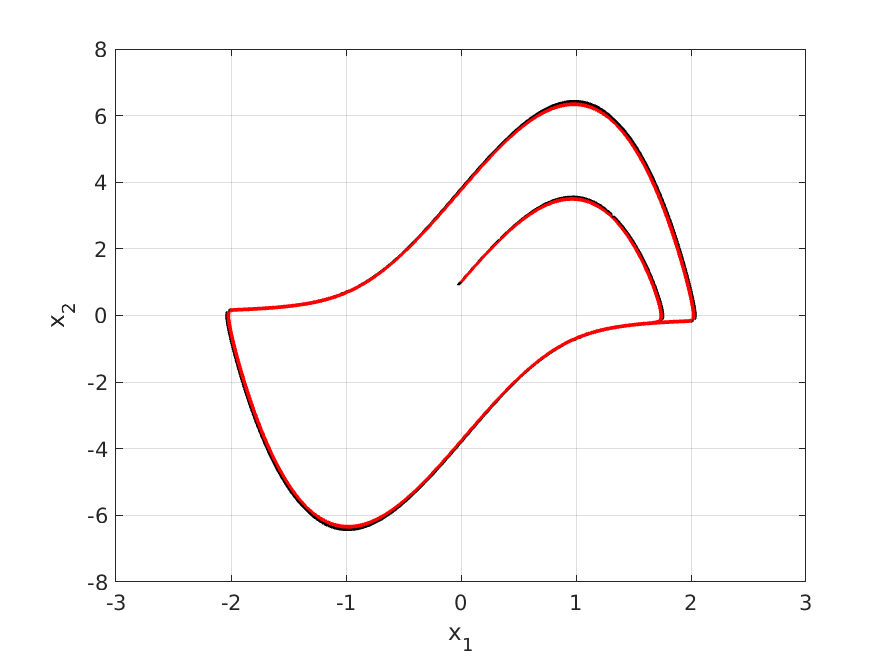} \tabularnewline
\end{tabular}\tabularnewline
\includegraphics[clip,width=1\textwidth]{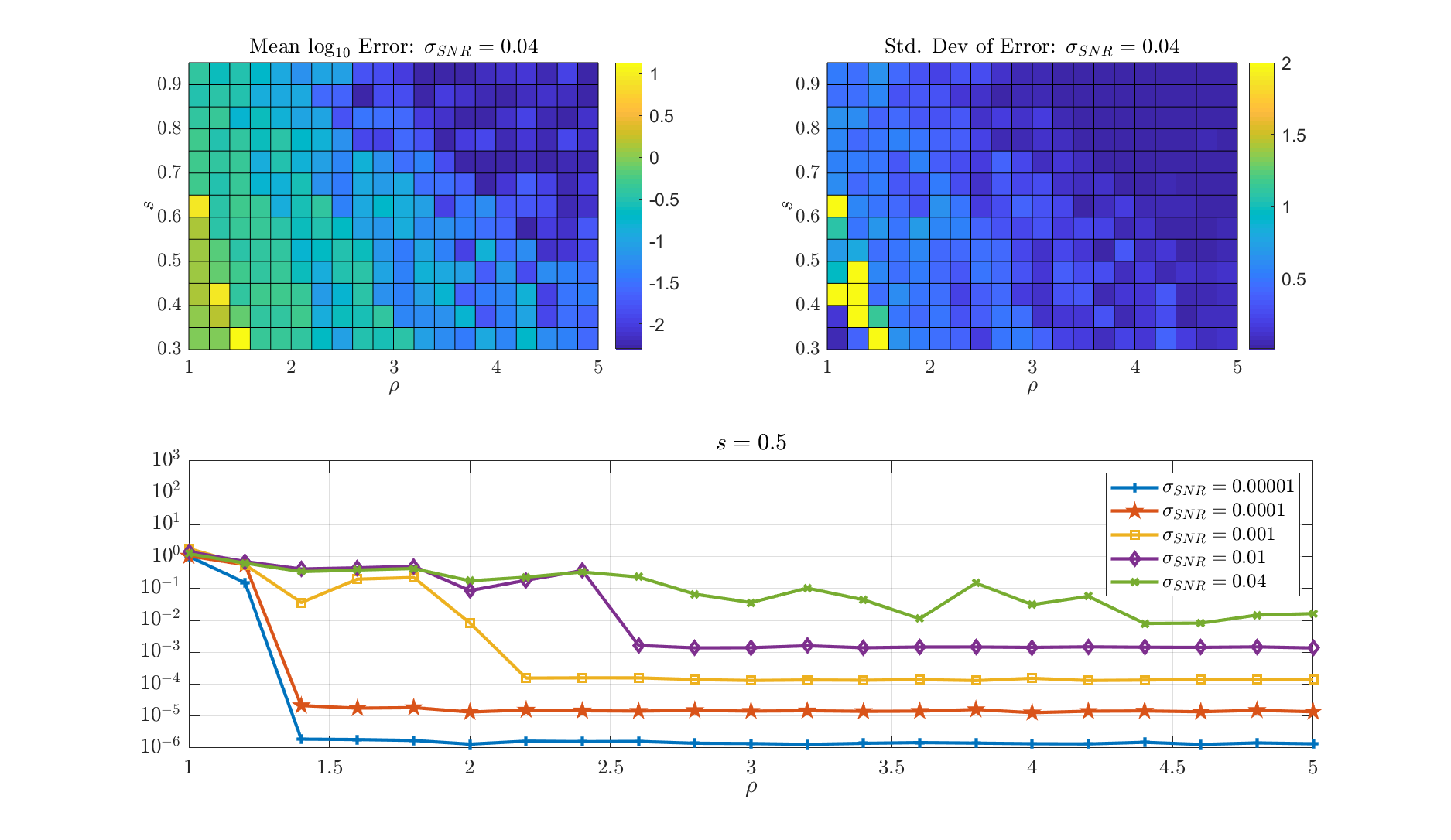} \tabularnewline
\end{tabular}\caption{Low-noise regime: dynamic recovery of the Van der Pol oscillator with
parameter $\mu=4$. Top row: example trajectory $\mathbf{y}$ (left)
and learned dynamics $\mathbf{x}_{dd}$ (right) both plotted over
true data $\mathbf{x}$. Here $\sigma_{SNR}=0.04$, $\rho=5$, $s=0.5$
and an error of $E:=\left\Vert \mathbf{\widehat{w}-\mathbf{w}^{\star}}\right\Vert _{2}/\left\Vert \mathbf{w}^{\star}\right\Vert _{2}=0.0026$.
Middle row: color plot of $\log_{10}$ of the average error $E$ (left)
and standard deviation of the error $E$ (right) over 200 noisy trajectories
with $\sigma_{SNR}=0.04$ with increasing $\rho$ along the $x$-axis
and increasing $s$ along the $y$-axis. Bottom: decreasing error
trend for fixed $s=0.5$ for various $\sigma_{SNR}$. As with the
Duffing equation in Figure \ref{lownz_duff}, for each $\sigma_{SNR}$
the expected error falls roughly an order of magnitude below the $\sigma_{SNR}$
as $\rho$ increases, however the expected accuracy begins to break
down for larger $\sigma_{SNR}$, motivating the use of WSINDy\_AG
for large noise.}
\label{lownz_vp} 
\end{figure}

\subsection{Large-Noise Regime}

Figures \ref{duff_hnz} to \ref{lorenz_hnz} show that Strategy 2
(non-uniform grid) can be employed to discover the dynamics in the
large noise regime. The signal to noise ratio is $\sigma_{SNR}=0.1$
for the Duffing, Van der Pol and Lorenz equations, and $\sigma_{SNR}=0.05$
for Lotka-Volterra. In each case we set the weak differentiation parameters
to $p=2$ and $s=16$ and the width-at-half-max to $r_{whm}=30$ timepoints.
For the 2D systems, we use $K=6J=126$ test basis functions, while
for the Lorenz equation $K=4J=224$ were used. In each case the correct
terms were identified with relative $\ell_{2}$ coefficient error
less than $10^{-2}$, indicating approximately two significant digits.
We plot the noisy data $\mathbf{y}$, the true data $\mathbf{x}$
and the simulated data-driven dynamical systems $\mathbf{x}_{dd}$
in dynamo view and phase space to exemplify the separation of scales
and the severity of the corruption from noise. We extend $\mathbf{x}_{dd}$
by $50\%$ to show that the data-driven system captures the expected
limiting behavior. Unless otherwise specified, we set the sparsity
and regularization parameters to $\lambda=\frac{1}{4}\min_{\mathbf{w}^{\star}\neq0}|\mathbf{w}^{\star}|$
and $\gamma=0$ and use a trial basis $(f_{j})_{j\in[J]}$ consisting
of all polynomials up to degree 5 in the state variables.\\

For the Duffing equation (Figure \ref{duff_hnz}), the data-driven
trajectory $\mathbf{x}_{dd}$ diverges slightly from the true data
as the system relaxes to equilibrium but is qualitatively correct.
The recovered Van der Pol oscillator (Figure \ref{vp_hnz}) identifies
the correct limit cycle but the dominant timescale of $\mathbf{x}_{dd}$
is slightly shorter than the true data $\mathbf{x}$. For this reason
$\mathbf{x}_{dd}$ slowly diverges pointwise over time but does not
stray from the relevant regions of phase space. This reflects that
more accurate measurements are needed to recover systems with abrupt
changes. The recovered Lotka-Volterra trajectory (Figure \ref{lv_hnz})
is nearly indistiguishable from the true data, but we note that here
regularization was used ($\gamma=0.01$), as the system was nearly
singular, and the columns of $\Theta(\mathbf{y})$ were normalized
using the 2-norm. For the Lorenz attractor (Figure \ref{lorenz_hnz}),
the recovered trajectory remains close to the true trajectory up until
around $t=3$, after which the two diverge as is expected from chaotic
dynamics. Nevertheless the Lorenz attractor is captured.

\begin{figure}[H]
\begin{tabular}{c}
\begin{tabular}{cc}
\includegraphics[clip,width=0.45\textwidth]{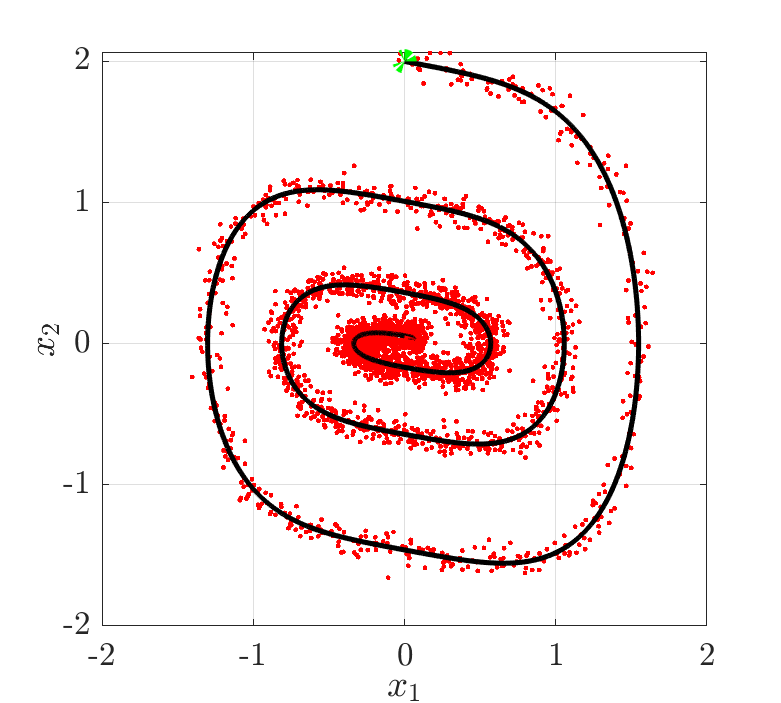}  & \includegraphics[clip,width=0.45\textwidth]{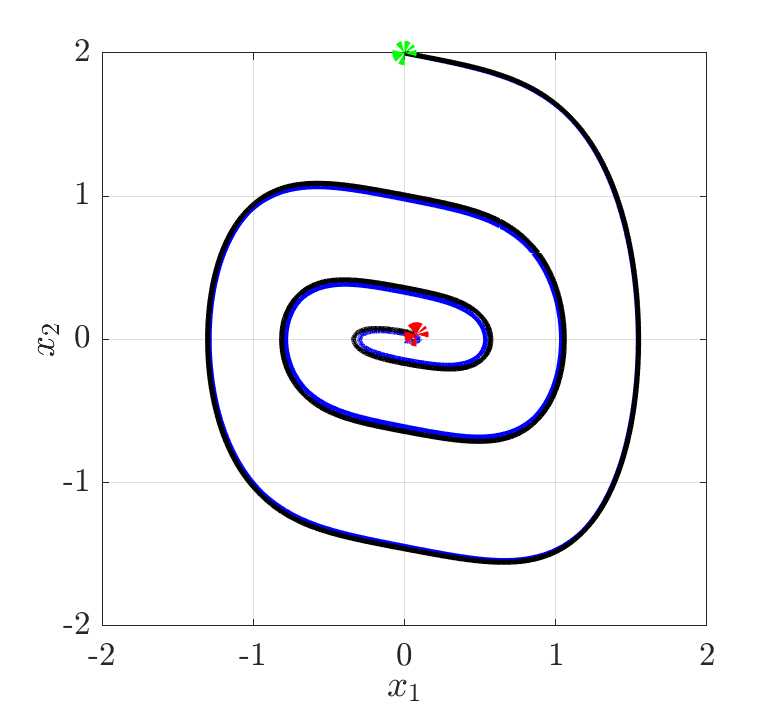} \tabularnewline
\end{tabular}\tabularnewline
\includegraphics[clip,width=0.95\textwidth]{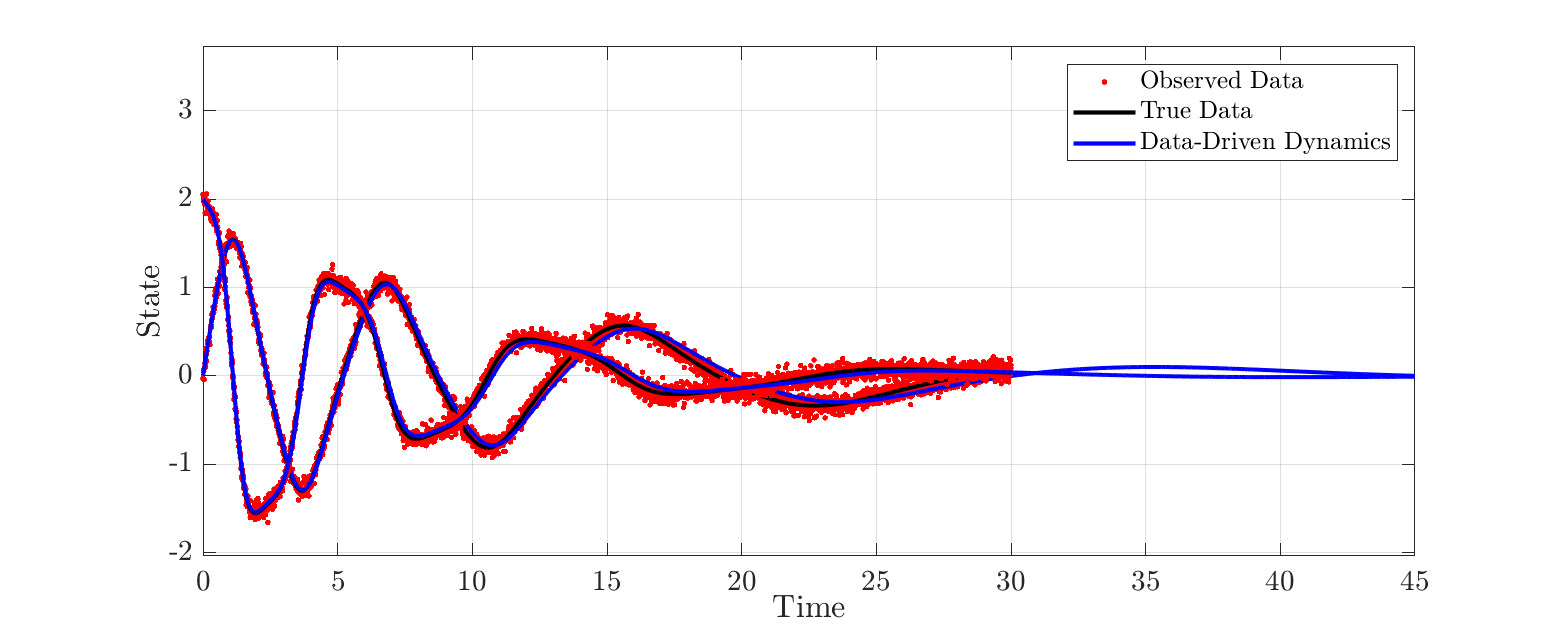} \tabularnewline
\tabularnewline
\end{tabular}\caption{Large-noise regime: Duffing Equation, same parameters as in Figure
\ref{lownz_duff}. Accurate recovery of the stable spiral with $\sigma_{SNR}=0.1$.
All correct terms were identified with an error in the weights of
$\left\Vert \mathbf{\widehat{w}-\mathbf{w}^{\star}}\right\Vert _{2}/\left\Vert \mathbf{w}^{\star}\right\Vert _{2}=0.007$
and $\left\Vert \mathbf{x}_{dd}-\mathbf{x}\right\Vert _{2}/\left\Vert \mathbf{x}\right\Vert _{2}=0.097$.
The number of basis functions used was $K=6J=126$ and the width-at-half-max
parameter was set to $r_{whm}=30$ time-points, resulting in $p=10$.}
\label{duff_hnz} 
\end{figure}

\begin{figure}[H]
\begin{tabular}{c}
\begin{tabular}{cc}
\includegraphics[clip,width=0.45\textwidth]{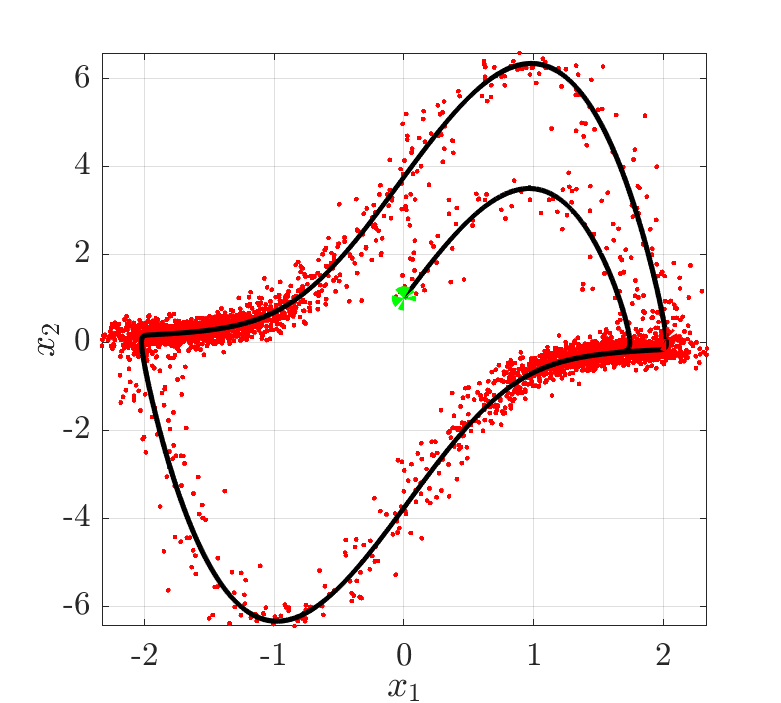}  & \includegraphics[clip,width=0.45\textwidth]{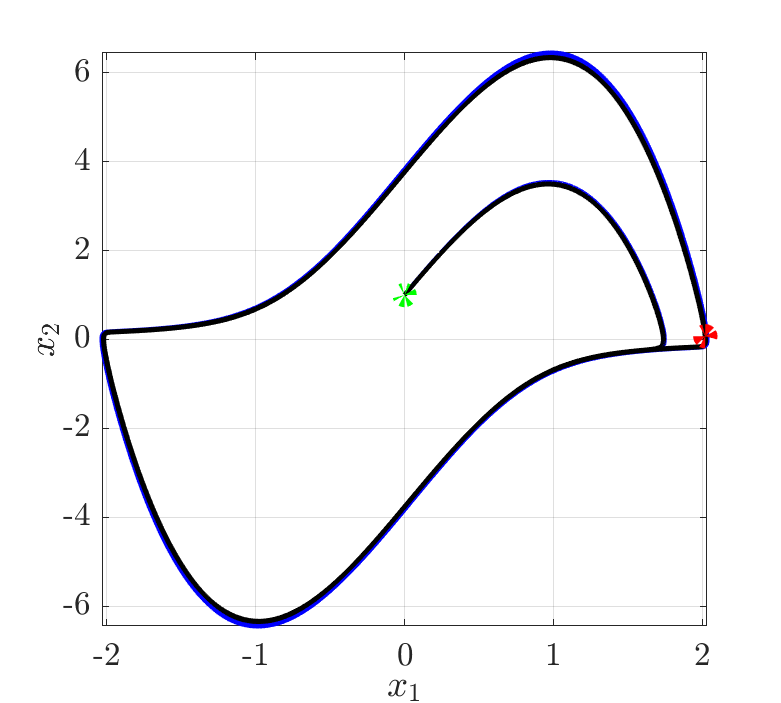} \tabularnewline
\end{tabular}\tabularnewline
\includegraphics[clip,width=0.95\textwidth]{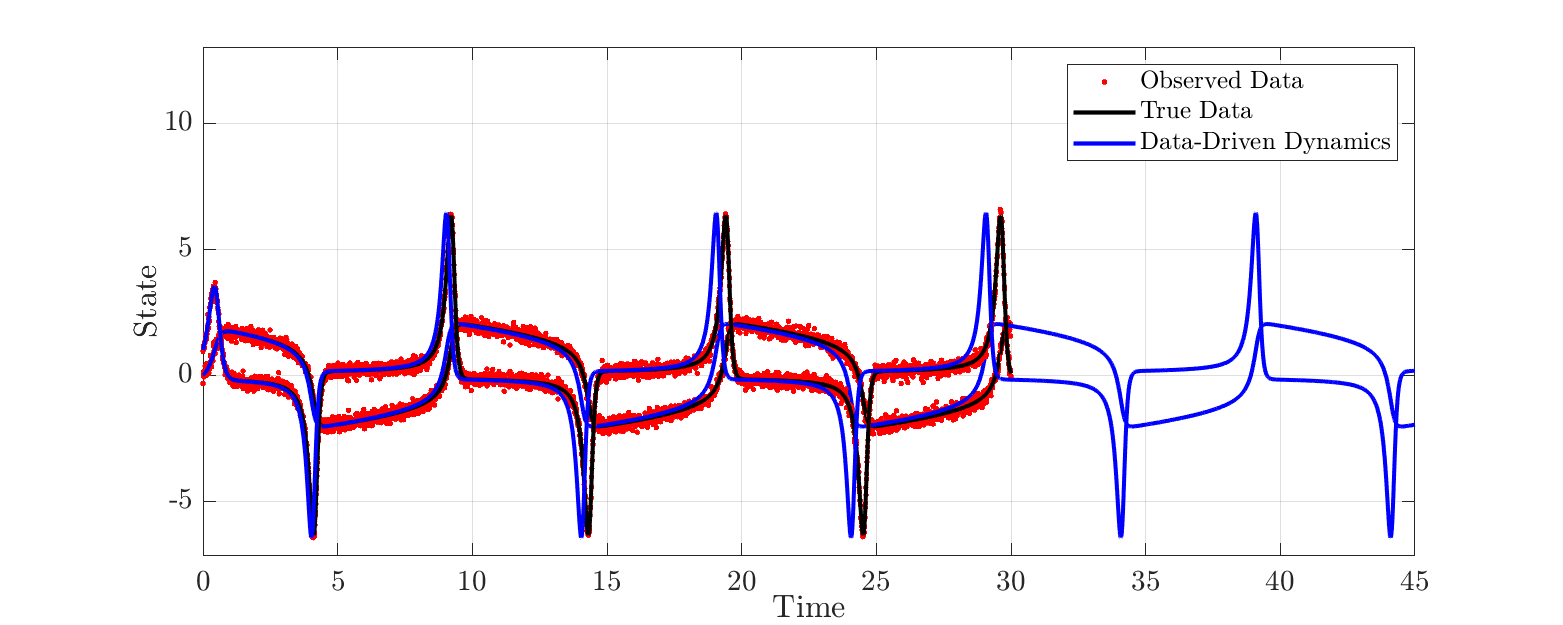} \tabularnewline
\tabularnewline
\end{tabular}\caption{Large-noise regime: Van der Pol oscillator, same parameters as in
Figure \ref{lownz_vp}. Accurate recovery of the limit cycle for $\sigma_{SNR}=0.1$.
All correct terms were identified with an error in the weights of
$\left\Vert \mathbf{\widehat{w}-\mathbf{w}^{\star}}\right\Vert _{2}/\left\Vert \mathbf{w}^{\star}\right\Vert _{2}=0.008$
and $\left\Vert \mathbf{x}_{dd}-\mathbf{x}\right\Vert _{2}/\left\Vert \mathbf{x}\right\Vert _{2}=0.56$.
The number of basis functions used was $K=6J=126$ and the width-at-half-max
parameter was set to $r_{whm}=30$ timepoints, resulting in $p=10$.
Here we see that the data-driven system $\mathbf{x}_{dd}$ traverses
the limit cycle with a slightly shorter period, resulting in a growing
pointwise error between the $\mathbf{x}_{dd}$ and the true state
$\mathbf{x}$.}
\label{vp_hnz} 
\end{figure}

\begin{figure}[H]
\begin{tabular}{c}
\begin{tabular}{cc}
\includegraphics[clip,width=0.45\textwidth]{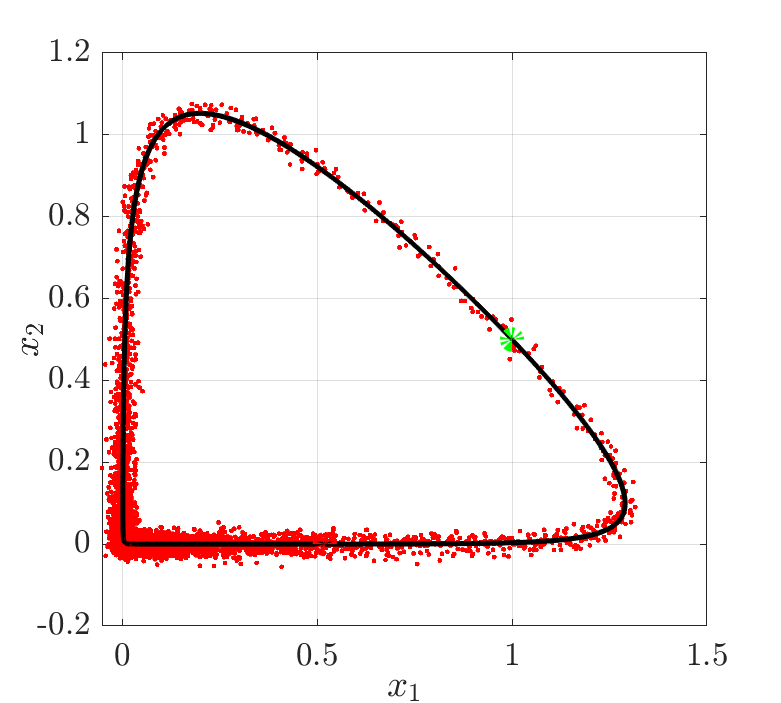}  & \includegraphics[clip,width=0.45\textwidth]{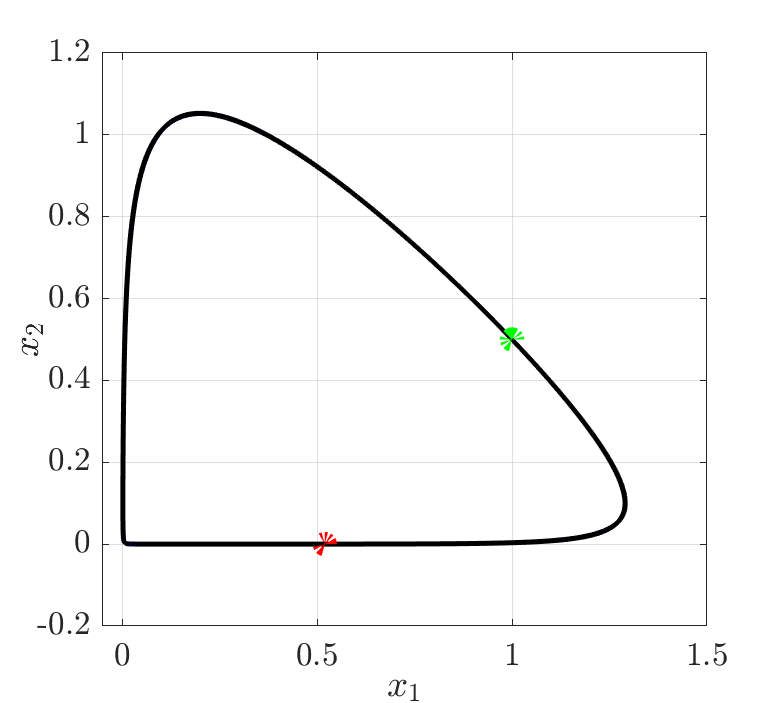} \tabularnewline
\end{tabular}\tabularnewline
\includegraphics[clip,width=0.95\textwidth]{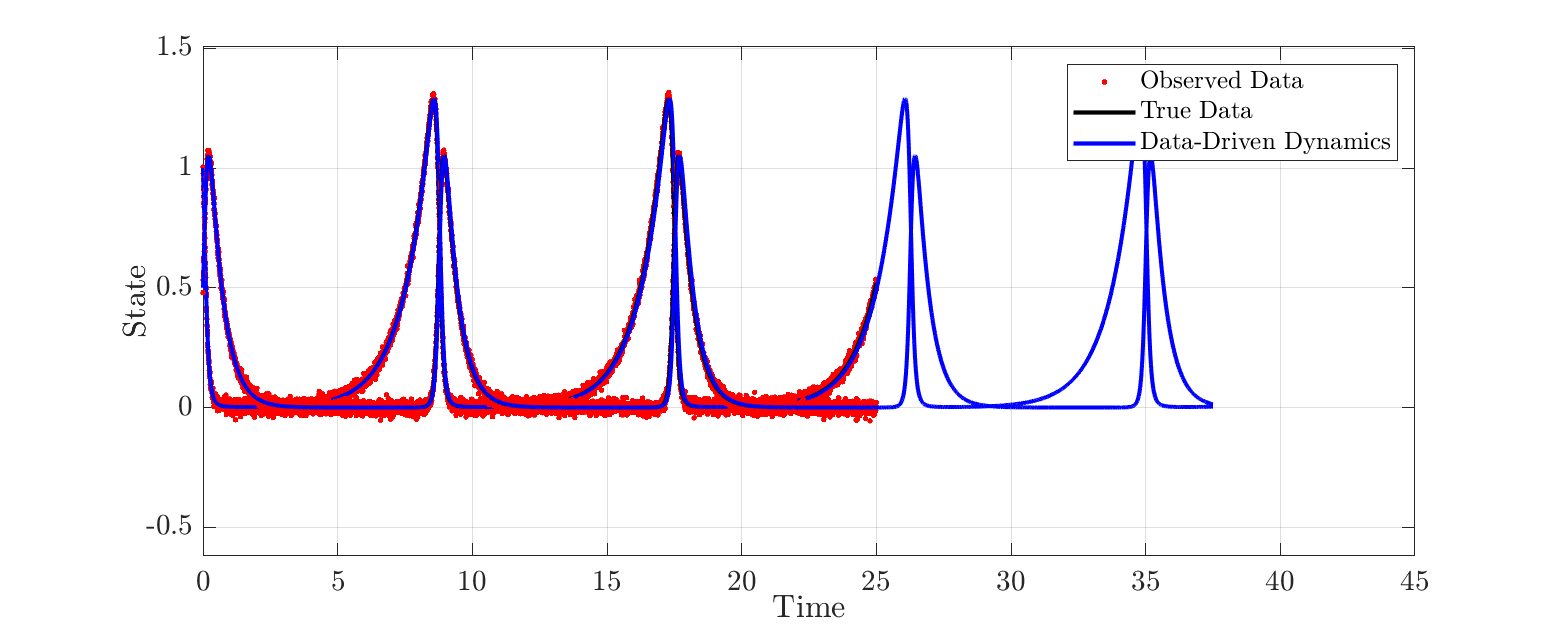} \tabularnewline
\tabularnewline
\end{tabular}\caption{Large-noise regime: Lotka-Volterra system with $\alpha=1,\beta=10$.
Accurate recovery of the limit cycle for $\sigma_{SNR}=0.05$. All
correct terms were identified with an error in the weights of $\left\Vert \mathbf{\widehat{w}-\mathbf{w}^{\star}}\right\Vert _{2}/\left\Vert \mathbf{w}^{\star}\right\Vert _{2}=0.0032$
and $\left\Vert \mathbf{x}_{dd}-\mathbf{x}\right\Vert _{2}/\left\Vert \mathbf{x}\right\Vert _{2}=0.065$.
The number of basis functions used was $K=6J=126$ and the width-at-half-max
parameter was set to $r_{whm}=30$ timepoints, resulting in $p=10$.
The data-diven system $\mathbf{x}_{dd}$ accurately captures the limit
cycle and traverses at the same speed as the true data $\mathbf{x}$.
Here we set $\gamma=0.01$ and normalized the columns of $\Theta(\mathbf{y})$
using the 2-norm as the original system is nearly linearly dependent.}
\label{lv_hnz} 
\end{figure}

\begin{figure}[H]
\begin{tabular}{c}
\begin{tabular}{cc}
\includegraphics[clip,width=0.45\textwidth]{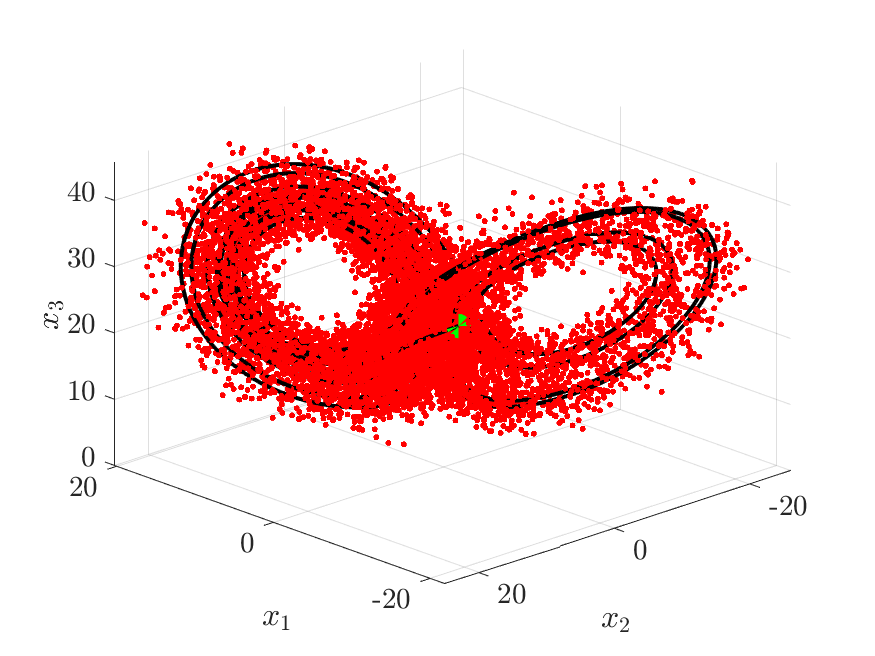}  & \includegraphics[clip,width=0.45\textwidth]{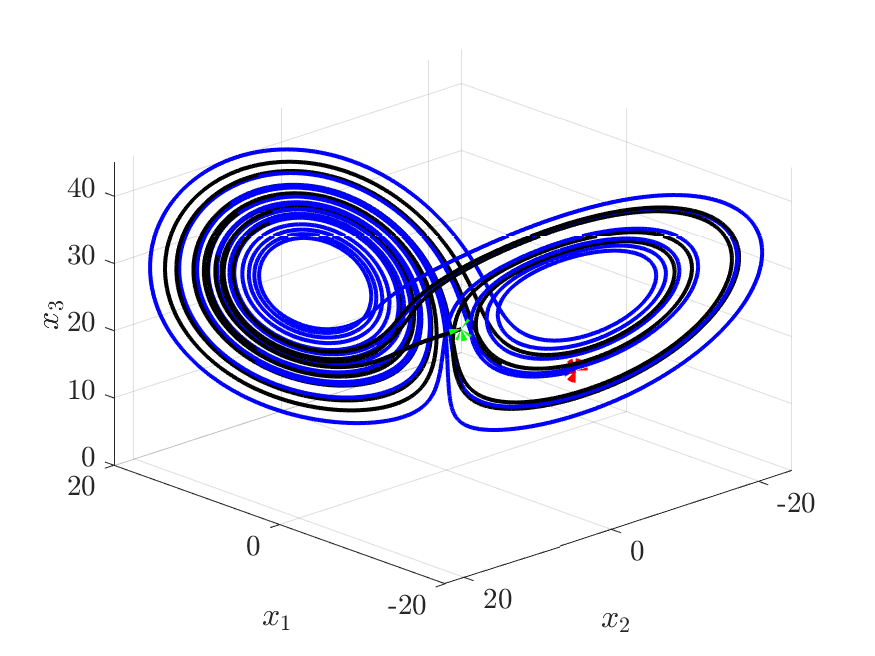} \tabularnewline
\end{tabular}\tabularnewline
\includegraphics[clip,width=0.95\textwidth]{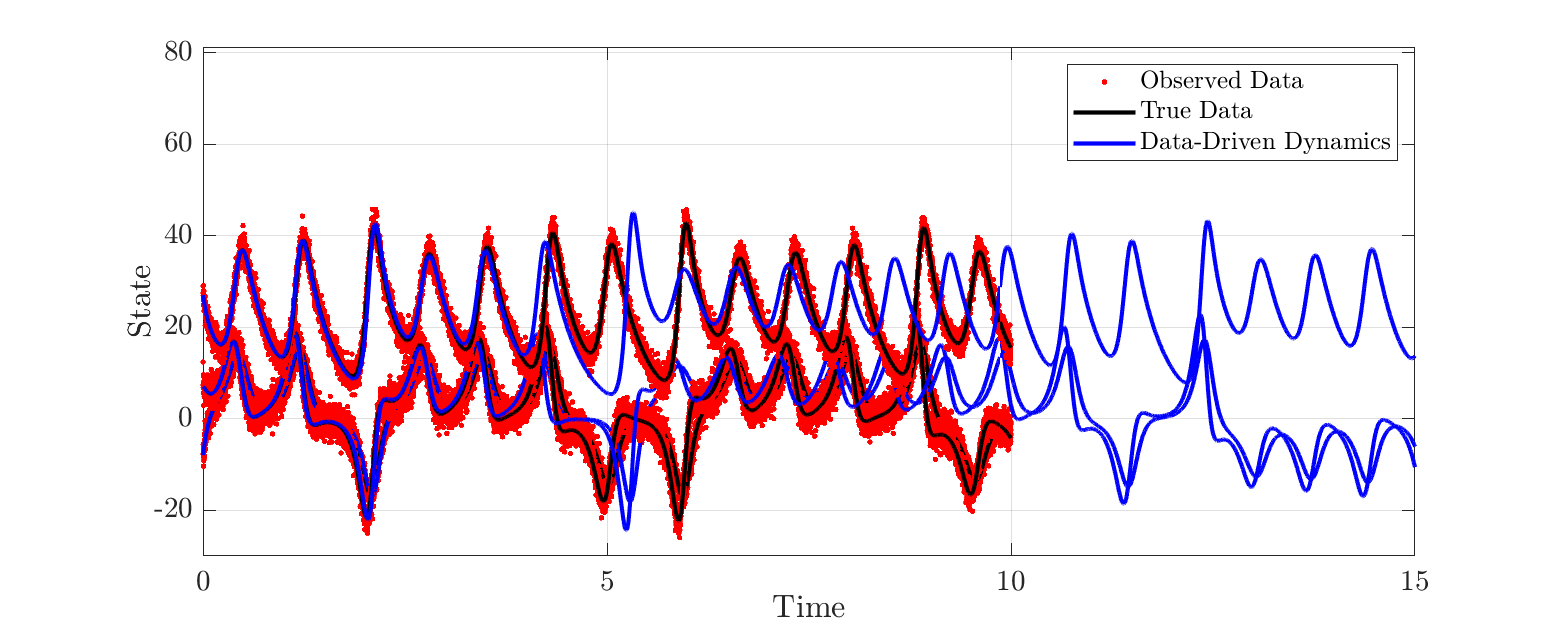} \tabularnewline
\tabularnewline
\end{tabular}
\caption{Large-noise regime: Lorenz system with $\mathbf{x}_{0}=[-8\ 7\ 27]^{T}$.
Accurate recovery of the strange attractor for $\sigma_{SNR}=0.1$.
All correct terms were identified with an error in the weights of
$\left\Vert \mathbf{\widehat{w}-\mathbf{w}^{\star}}\right\Vert _{2}/\left\Vert \mathbf{w}^{\star}\right\Vert _{2}=0.0091$.
Since the system is chaotic it does not make sense to measure the
pointwise error between the data-driven trajectory $\mathbf{x}_{dd}$
and the true data $\mathbf{x}$ for large times, but using data up
until $t=3$ (first 3000 timepoints) we can report a reasonable pointwise
agreement of $\left\Vert \mathbf{x}_{dd}-\mathbf{x}\right\Vert _{2}/\left\Vert \mathbf{x}\right\Vert _{2}=0.047$.
The number of basis functions used was $K=4J=224$ and the width-at-half-max
parameter was set to $r_{whm}=30$ timepoints, resulting in $p=10$.}
\label{lorenz_hnz} 
\end{figure}

\section{Concluding Remarks}

We have developed and investigated a data-driven model selection algorithm
based on the weak formulation of differential equations. The algorithm
utilizes the reformulation of the model selection problem as a sparse
regression problem for the weights $\mathbf{w}^{\star}$ of a candidate
function basis $(f_{j})_{j\in[J]}$ introduced in \cite{wang2011predicting}
and generalized in \cite{brunton2016discovering} as the SINDy algorithm.
Our Weak SINDy algorithm (WSINDy) can be seen as a generalization
of the sparse recovery scheme using integral terms found in \cite{schaeffer2017sparse},
where dynamics were recovered from the integral equation.

A natural line of inquiry is to consider how WSINDy compares with
conventional SINDy. There are several notable advantages of WSINDy;
in particular, by considering a weak form of the equations, WSINDy
completely avoids any evaluation of the pointwise derivatives which
cause significant accuracy problems in conventional SINDy. While our
work here is not the first to consider integrals of the differential
equation,\footnote{To the best of the author's knowledge \cite{schaeffer2017sparse}
was the first (and only other) such effort.} WSINDy is a much more general methodology with natural extensions
that draw upon the rich and well-established finite element literature.
Furthermore, the integration operation itself is a linear operator,
allowing estimation of covariance structures, which allow computation
of solutions in a weighted least squares framework (which is substantially
more accurate that ordinary least squares).

In our future work, we also plan to pursue direct computational comparisons
with conventional SINDy to illustrate concrete benefits of the WSINDy
framework. For example, our preliminary efforts (not reported here)
suggest that the relationship between noise and coefficient error
appears to be more predictable in WSINDy than conventional SINDy.
As Figures \ref{lownz_duff} and \ref{lownz_vp} suggest, WSINDy appears
to scale continuously with the SNR and without the use of noise filtering,
a trend which simply cannot be expected from pointwise derivative
approximations. Our preliminary efforts also suggest that the linear
systems are smaller, requiring fewer data points, but more overall
floating point operations. Clearly a more complete analysis is required
to fully understand the tradeoffs between WSINDy and SINDy.

Lastly, the most obvious extensions lie in generalizing the WSINDy
method for spatiotemporal datasets and problems with multiple timescales.
WSINDy as presented here in the context of ODEs is an exciting proof
of concept, with natural extensions to spatiotemporal and multiresolution
settings building upon the extensive results in numerical and functional
analysis for weak and variational formulations of physical problems.

\section{Acknowledgements}

This research was supported in part by the NSF/NIH Joint DMS/NIGMS
Mathematical Biology Initiative grant R01GM126559 and in part by the
NSF Computing and Communications Foundations Division grant CCF-1815983. Code used in this
manuscript is publicly available on GitHub at \url{https://github.com/dm973/WSINDy}.

\bibliographystyle{plain}
\bibliography{researchCU}

\end{document}